\documentclass[a4paper,10pt]{article}

\usepackage{amsthm}
\usepackage{amsmath}
\usepackage{amsxtra}
\usepackage{amssymb}
\usepackage{thmtools}
\usepackage[colorlinks=true,linkcolor=blue,citecolor=blue,urlcolor=blue]{
hyperref}
\usepackage{mathrsfs}
\usepackage{turnstile}
\usepackage{comment}
\usepackage{enumitem}
\usepackage[retainorgcmds]{IEEEtrantools}

\usepackage[all]{xy}
\usepackage{graphics}
\usepackage{graphicx}
\usepackage{tikz}
\usetikzlibrary{matrix,positioning,calc,arrows,arrows.meta}
\usetikzlibrary{decorations.pathmorphing,shapes}

\declaretheorem[name=Definition,style=definition,qed=$\dashv$,
numberwithin=section]{dfn}
\declaretheorem[name=Example,style=definition,sibling=dfn]{exm}
\declaretheorem[name=Theorem,style=plain,sibling=dfn]{tm}
\declaretheorem[name=Theorem,style=plain,numbered=no]{tm*}
\declaretheorem[name=Fact,style=plain,sibling=dfn]{fact}
\declaretheorem[name=Fact,style=plain,numbered=no]{fact*}
\declaretheorem[name=Lemma,style=plain,sibling=dfn]{lem}

\declaretheorem[name=Corollary,style=plain,sibling=dfn]{cor}
\declaretheorem[name=Remark,style=definition,sibling=dfn]{rem}
\declaretheorem[name=Claim,style=plain]{clm}
\declaretheorem[name=Claim,style=plain,numbered=no]{clm*}

\declaretheorem[name=Sublaim,style=plain,numbered=no]{sclm*}

\newcommand{\Lim}{\mathrm{Lim}}

\newcommand{\Eee}{\mathscr{E}}

\newcommand{\iso}{\cong}

\newcommand{\sub}{\subseteq}
\newcommand{\cross}{\times}
\newcommand{\all}{\forall}

\newcommand{\inter}{\cap}
\renewcommand{\int}{\inter}

\newcommand{\om}{\omega}
\newcommand{\pow}{\mathcal{P}}
\newcommand{\OR}{\mathrm{OR}}

\newcommand{\Hull}{\mathrm{Hull}}

\newcommand{\cut}{\backslash}

\newcommand{\Ll}{\mathcal{L}}

\newcommand{\rg}{\mathrm{rg}}
\newcommand{\dom}{\mathrm{dom}}

\newcommand{\crit}{\mathrm{crit}}

\newcommand{\rest}{\!\upharpoonright\!}
\newcommand{\com}{\circ}

\newcommand{\Ult}{\mathrm{Ult}}

\newcommand{\sats}{\models}
\newcommand{\elem}{\preccurlyeq}

\newcommand{\AC}{\mathsf{AC}}

\newcommand{\HOD}{\mathrm{HOD}}

\newcommand{\ZFC}{\mathsf{ZFC}}
\newcommand{\ZF}{\mathsf{ZF}}

\newcommand{\id}{\mathrm{id}}

\newcommand{\trancl}{\mathrm{trancl}}

\newcommand{\scot}{\mathrm{scot}}

\newcommand{\psub}{\subsetneq}

\newcommand{\lpole}{\left\lfloor}
\newcommand{\rpole}{\right\rfloor}
\newcommand{\univ}[1]{\lpole #1\rpole}
\newcommand{\tu}{\textup}

\DeclareMathOperator{\Th}{Th}

\DeclareMathOperator{\cof}{cof}

\DeclareMathOperator{\rank}{rank}

\begin{document}
\date{November 15, 2020}

\title{Periodicity in the cumulative hierarchy}

\author{Gabriel Goldberg\footnote{G. Goldberg is supported by NSF Grant DMS 1902884.} 
\hspace{.15cm}\& Farmer Schlutzenberg\footnote{
F.~Schlutzenberg is supported by Deutsche Forschungsgemeinschaft (DFG, German 
Research
Foundation) under Germany's Excellence Strategy EXC 2044-390685587,
Mathematics M\"unster: Dynamics-Geometry-Structure.}
}

\maketitle

\begin{abstract}We investigate the structure
of rank-to-rank elementary embeddings, working in $\ZF$ set theory without 
the  Axiom of Choice.
Recall that the levels $V_\alpha$ of the cumulative hierarchy are 
defined via iterated application of the power set operation, 
starting from $V_0=\emptyset$, and taking unions at limit stages.
Assuming that 
\[ j:V_{\alpha+1}\to V_{\alpha+1}\]
is a (non-trivial) elementary embedding,
we show that the structure of $V_\alpha$ is fundamentally
different to that of $V_{\alpha+1}$. 
We show that $j$
is definable from parameters over 
$V_{\alpha+1}$ iff $\alpha+1$
is an odd ordinal. 
Moreover,
if $\alpha+1$ is odd then $j$ is definable over $V_{\alpha+1}$
from the parameter
\[ j`` V_{\alpha}=\{j(x)\bigm|x\in V_\alpha\},\] and 
uniformly so.
This parameter is optimal in that
$j$ is not definable from any parameter which is an element of $V_\alpha$.
In the 
case that $\alpha=\beta+1$, we also give a characterization of such $j$ in 
terms of ultrapower maps via certain ultrafilters.

Assuming $\lambda$ is a limit ordinal,
we prove that if
$j:V_\lambda\to V_\lambda$ is $\Sigma_1$-elementary,
then $j$ is not definable over $V_\lambda$
from parameters,
and if $\beta<\lambda$ and $j:V_\beta\to V_\lambda$ is
fully elementary and $\in$-cofinal, then $j$ is likewise not definable;
note that this last result is relevant to embeddings of much lower consistency 
strength than rank-to-rank.

If there is a Reinhardt cardinal,
then for all 
sufficiently large ordinals $\alpha$,
there is indeed an elementary $j:V_\alpha\to V_\alpha$,
and therefore the cumulative hierarchy is 
eventually \emph{periodic} (with 
period 2).

\end{abstract}

\section{Introduction}

The universe $V$ of all sets
is the union of the 
\emph{cumulative hierarchy} \(\langle V_\alpha\rangle_{\alpha\in \OR}\).
Here $\OR$ denotes the class of all ordinals,
and the sets $V_\alpha$ are obtained by iterating the power 
set operation $X\mapsto\pow(X)$ transfinitely, 
starting with $V_0=\emptyset$,
setting $V_{\alpha+1}=\pow(V_\alpha)$,
and $V_\eta=\bigcup_{\alpha<\eta}V_\alpha$ for limit ordinals $\eta$.

Before Cantor's discovery of the transfinite ordinals,
mathematicians typically only considered sets lying
quite low in the infinite levels of the cumulative hierarchy (below 
\(V_{\omega+5}\) say).
Since then our understanding much higher in the hierarchy
has deepened extensively. It is possible
to take the view, however, that most research has been 
focused below a certain threshold,
due to its interaction with the Axiom of 
Choice.
This 
paper investigates certain features of the hierarchy which first appear just 
beyond this threshold.

After some distance, finite intervals in the cumulative hierarchy have the 
appearance of uniformity:
for large infinite limit ordinals \(\gamma\) and large natural numbers 
\(n\) and \(m\), one might expect not to find natural set theoretic properties
which differentiate between \(V_{\gamma+n}\) and 
\(V_{\gamma+m}\):
one might expect  \(V_{\gamma+813}\), 
for example, to be essentially structurally indistinguishable from 
\(V_{\gamma+814}\).
But the key result of this paper shows that assuming $\gamma$ is 
\emph{very large} --- so large, in fact, that the Axiom 
of Choice must be violated ---
\(V_{\gamma+813}\) and \(V_{\gamma+814}\) display fundamental structural 
differences.
More generally, the properties of \(V_{\gamma+n}\) 
depend the parity of \(n\).

Exactly how large must $\gamma$ be 
for these differences to arise?
To answer this question requires introducing some 
basic concepts from the theory of \emph{large cardinals}, 
one of the main areas of research in modern set theory.
The simplest example of a large cardinal\footnote{There is no general
formal definition of ``large cardinal''.} is
an \emph{inaccessible 
cardinal}. An uncountable ordinal $\kappa$ is inaccessible if
every function from \(V_\alpha\) to \(\kappa\) where
\(\alpha < \kappa\) is bounded strictly below 
\(\kappa\).\footnote{An ordinal $\alpha$ is formally equal to
the set of ordinals $\beta<\alpha$, so if $\pi:X\to\kappa$,
then $\pi$ is bounded strictly below $\kappa$
iff there is $\alpha<\kappa$ such that $\pi(\beta)<\alpha$
for all $\beta\in X$.}\footnote{Assuming 
the Axiom of Choice $\AC$,
inaccessibility is usually
defined slightly differently, but under $\AC$,
the definitions are equivalent. The definition we give
here is the appropriate one when one does not assume $\AC$.}
So inaccessible cardinals are ``unreachable from below'',
and form a natural kind of closure point of the set theoretic universe.
 If $\kappa$ is inaccessible then $V_\kappa$ models all of the $\ZF$ axioms,
as does $V_\alpha$ for unboundedly many ordinals $\alpha<\kappa$.
So by G\"odel's Incompleteness Theorem, inaccessible cardinals  cannot 
be proven to exist in \(\ZF\),
and  inaccessibility somehow ``transcends'' $\ZF$.
(The \emph{Zermelo-Fr\"ankel} axioms, denoted $\ZF$,
are the usual axioms of set theory, without the Axiom of Choice $\AC$. And 
$\ZFC$ denotes $\ZF$ augmented with $\AC$.)

Inaccessibles are just the beginning. Further up in the hierarchy, large 
cardinals are typically exhibited by some
form of non-identity \emph{elementary embedding}
\[ j:V\to M\]
from the universe $V$ of all sets to
 some 
transitive\footnote{That is,  for all $x\in M$, we 
have $x\sub M$.} class
$M\sub V$. \emph{Elementarity} demands that $j$ preserve the truth
of all first-order statements in parameters 
between $V$ and $M$ (see \S\ref{subsec:terminology} for details).
One can show that
 there is an ordinal $\kappa$ such that
$j(\kappa) > \kappa$, and the least such ordinal is called the \emph{critical 
point} $\crit(j)$ of $j$; if $\ZFC$\footnote{Under $\ZFC$, this notion is 
equivalent to measurability, but the notions
are not equivalent in general under $\ZF$ alone.} holds then such a critical 
point is known as a \emph{measurable 
cardinal}. 
The critical point of an elementary embedding is inaccessible,
and in fact there are unboundedly many inaccessible cardinals $\eta<\kappa$. 
So such critical points transcend 
inaccessible cardinals. Critical points are transcended by still larger 
large cardinals.

Large cardinal axioms
are by far the most widely accepted and well-studied 
principles extending the standard axioms of set theory.\footnote{An example of a large cardinal 
axiom is the assertion that there is an inaccessible cardinal or
the assertion that there is a critical point cardinal.
While there is no formal definition of the term ``large cardinal axiom'',
there is little controversy over which principles qualify as large cardinal axioms.} 
One of the main reasons for this is the empirical fact
that large cardinal axioms are arranged in an essentially
linear hierarchy of strength,
with each large cardinal notion typically transcending all the preceding 
ones.\footnote{This is a bit of an oversimplification.}
There is no known example of a pair of incompatible large cardinal axioms,
and the linearity phenomenon suggests that none will ever arise.

The strength of a large cardinal notion $j:V\to M$ depends in large part on the 
extent to 
which $M$ resembles $V$ and contains fragments of $j$. 
So taking the notion
to its logical extreme, William Reinhardt suggested in his dissertation
taking $M=V$; that is, a (non-identity) elementary embedding
\[ j:V\to V. \]
The critical point of such an embedding became known as a \emph{Reinhardt 
cardinal}.
But Kunen proved in \cite{kunen_no_R} (see also \cite{dimonte} and 
\cite{gen_kunen_incon}) that, assuming
$\ZFC$, they do not exist.
In fact, suppose $j:V\to M$ is elementary where $M\sub V$ is a transitive 
class and $j$ is not the identity. 
Letting 
$\kappa_0=\crit(j)$ and
$\kappa_{n+1}=j(\kappa_n)$, then because $j$ is order-preserving on ordinals 
(an easy
consequence of elementarity),
\[ \kappa_0<\kappa_1<\ldots<\kappa_n<\ldots. \]
Let their supremum be
$\lambda=\sup_{n<\om}\kappa_n$.
We write $\kappa_{n}(j)=\kappa_n$ and $\kappa_{\om}(j)=\lambda$.
Kunen proved in \cite{kunen_no_R} (from $\ZFC$) that
 $V_{\lambda+1}\not\sub M$. He also proved that there is no ordinal
$\lambda'$ and elementary embedding
\[ j:V_{\lambda'+2}\to V_{\lambda'+2}.\]
So $\AC$ enforces a rather
abrupt 
upper limit to the large cardinal hierarchy.

But it has remained
a mystery whether $\AC$ is actually needed to 
prove
there can be no elementary $j:V\to V$.
Suzuki \cite{suzuki_no_def_j} showed in \(\ZF\) alone that
such a $j$ \emph{cannot} be definable from parameters
over $V$.
This leads to a metamathematical question: what exactly is a class?
In the most restrictive formulation,
classes are all definable from parameters,
so in this setting, Suzuki's result rules out an elementary $j:V\to V$
from $\ZF$ alone, and the matter is settled  -- though not the 
$j:V_{\lambda+2}\to V_{\lambda+2}$ matter,
which is immune to Suzuki's argument. But one can also formulate 
classes 
more generally,
and appropriately formulated,
there is no known way to 
disprove the 
existence of $j:V\to V$ without $\AC$. For the most part in this paper,
we  focus anyway on embeddings of set size, so the precise
definition of classes is not so important for us here.\footnote{In 
\S\ref{sec:large_enough} we will deal with actual Reinhardt
cardinals, and will mention an appropriate formulation of classes there.} 

Note that one can state Kunen's result from a different angle:
if $j:V\to V$ is elementary and 
$\lambda=\kappa_\om(j)$,
then there is a failure of $\AC$
within $V_{\lambda+2}$.
In this sense, very strong elementary 
embeddings limit the extent
to which $\AC$ can be valid, and set theory under its assumption
can be seen as focusing on sets inside $V_\lambda$,
below the threshold where $\AC$ breaks down.

In the last few years, there has been growing 
interest in investigating large cardinal notions like $j:V\to V$ and beyond,
assuming  $\ZF$ or second order $\ZF$, often augmented with fragments of 
$\AC$ (and also large cardinal notions below this level,
but without assuming $\AC$).\footnote{See for 
example  
\cite{woodin_koellner_bagaria}, \cite{sarg_apt_jonsson_V_to_V},
\cite{cutolo_structure}, 
 \cite{con_lambda_plus_2},
 \cite{critical_cardinals}, 
\cite{short_note_vlc},
\cite{goldberg_even_numbers},
\cite{suzuki_no_def_j},
 \cite{usuba_ls},
\cite{cutolo_cofinality},
 \cite{super_rein},
 \cite{extenders_ZF}.}
 This paper sits 
within that line of investigation,  just beyond the level which 
violates choice, focusing on elementary, or 
at least $\Sigma_1$-elementary,\footnote{That is, 
$V_\alpha\sats\varphi(\vec{x})$ iff $V_\alpha\sats\varphi(j(\vec{x}))$
for all $\Sigma_1$ formulas $\varphi$ and $\vec{x}\in 
V_\alpha^{<\om}$.} embeddings of the form\[ j:V_\alpha\to 
V_\alpha\]
with $\alpha$ an ordinal. Generalizing some standard terminology,
we call these
\emph{rank-to-rank} 
embeddings,\footnote{In the $\ZFC$ context, by Kunen's Theorem,
the only rank-to-rank embeddings in this strict sense are $k:V_\lambda\to 
V_\lambda$
or $k:V_{\lambda+1}\to V_{\lambda+1}$ where $\lambda=\lambda(k)$
(his proof does rule out a $\Sigma_1$-elementary $k:V_{\lambda+2}\to 
V_{\lambda+2}$).
The $I_0$ embeddings $j:L(V_{\lambda+1})\to L(V_{\lambda+1})$
are also traditionally known as \emph{rank-to-rank} embeddings,
even if the terminology does not seem to quite match reality in that case.
We adopt the same \emph{rank-to-rank} terminology for $\Sigma_1$-elementary 
$j:V_\alpha\to 
V_\alpha$ in general because it is very natural.}
because $V_\alpha$ is a \emph{rank initial segment}
of $V$. If there is a Reinhardt cardinal
then there is an ordinal $\lambda$ such that for all $\alpha\geq\lambda$,
there is an elementary $j:V_\alpha\to V_\alpha$;
see Theorem \ref{tm:lambda_rank-Berkeley}.

We primarily consider the following question, with $\ZF$ as background theory.
Let $\alpha$ be an 
ordinal and $j:V_\alpha\to V_\alpha$ be elementary. Is $j$ definable 
from parameters over $V_\alpha$? That is, we investigate whether there is $p\in 
V_\alpha$
and some formula $\varphi$ in the language of set theory
(with binary predicate symbol $\in$ for membership) such that for all 
$x,y\in V_\alpha$, we have
\[ j(x)=y\iff V_\alpha\sats\varphi(p,x,y),\]
where $\sats$ is the usual model theoretic truth satisfaction
relation.

It turns out that there is a very simple answer to this question,
generalizing Suzuki's theorem, but with a twist.
We say that an ordinal $\alpha$ is \emph{even} iff $\alpha=\eta+2n$ for some
$n<\om$, with $\eta=0$ or $\eta$ a limit ordinal.
Naturally, \emph{odd} means not even.

\begin{tm}\label{tm:cumulative_periodicity}\footnote{This theorem
is also proved in \cite{goldberg_even_numbers}, where the theorem
is then applied in generalizing  Woodin's $I_0$ theory. In the present paper,
we focus on Theorem \ref{tm:cumulative_periodicity}
and closely related results, some of which are lemmas toward its 
proof, and some of which extend it. There is 
more discussion of those at the end of this introduction.}
 Let $j:V_\alpha\to V_\alpha$ be fully elementary, with $j\neq\id$.
 Then $j$ is definable from parameters over $V_\alpha$
 iff $\alpha$ is odd.
\end{tm}

The proof appears at the end of \S\ref{sec:def_rank-to-rank},
and then a second, slightly different proof is sketched in Remark 
\ref{rem:proof_2_cumulative_periodicity}.

So if there is an elementary
$j:V_{\eta+184}\to V_{\eta+184}$ (and hence an elementary embedding from $V_{\eta+183}$ to 
$V_{\eta+183}$, namely \(j\restriction V_{\eta+183}\)),
then $V_{\eta+183}$ and $V_{\eta+184}$ are indeed different
(but $V_{\eta+182}$ analogous to $V_{\eta+184}$, etc).
The proof will also yield much more information about such embeddings,
and in the successor case, give a characterization of them,
and reveal strong structural 
differences between the odd and even levels
which admit such embeddings.
A consequence of Theorem \ref{tm:lambda_rank-Berkeley} will also be that if 
there is a 
Reinhardt cardinal, and $j:V\to V$,
then \emph{all} ordinals $\eta\geq\kappa_{\om}(j)$ are indeed large enough 
for this \emph{periodicity} phenomenon to take hold.

Periodicity phenomena (with period 2) 
are of course a familiar feature of logical quantifiers:
$\all x_0\exists y_0\all x_1\exists y_1\ldots$
They are pervasive in descriptive 
set theory
(in particular in the Periodicity Theorems, see \cite{mosch}).
But in such cases, which arise in the analysis of complexity classes
and so forth arising from quantifier alternation, the 
periodicity is built into the definitions in the first place. 
This particular instance of periodicity shows up more subtly in \emph{inner 
model theory},
in particular regarding the canonical inner model $M_n$ with $n$
\emph{Woodin cardinals}, where $n$ is finite;\footnote{$M_0$ is just G\"odel's 
constructible universe $L$.} Woodin cardinals are beyond measurables, 
but well below those we consider in this paper. It turned out that $n$ Woodin
cardinals corresponds tightly to $n$ alternations of quantifiers
over real numbers, and this has the result that many important features
of $M_n$ depend on the parity of $n$.
However, the basic definition of $M_n$ (and similarly for $n$ measurable 
cardinals etc) does not have any obvious dependence on parity built into it.
The periodicity
present in Theorem \ref{tm:cumulative_periodicity} is in this sense
analogous
to the case of $M_n$.
In both  cases  just 
mentioned
and Theorem \ref{tm:cumulative_periodicity}, there are 
stark differences
between the even and odd sides. 
The periodicity in the $V_\alpha$'s also seems to manifest certain
``$\all$/$\exists$'' features, although the full nature of this is probably as 
of yet  not understood.

In \S\ref{sec:reinhardt_filters}
we present a different perspective on elementary
$j:V_{\alpha+2}\to 
V_{\alpha+2}$,
relating such elementary embeddings
to ultrapower embeddings via associated ultrafilters, and sketch
the proof
of Theorem \ref{tm:cumulative_periodicity} for successor ordinals again, from 
this new perspective.
We also establish a characterization of 
such $j$
 in terms of ultrapower embeddings.
\footnote{There is an important subtlety here.
We will identify a certain ultrafilter $U$ and form the ultrapower 
$U=\Ult(V_{\alpha+2},U)$, and define $i:V_{\alpha+2}\to U$ to be the 
ultrapower map. We will show that $i=j$, i.e., these maps have the same 
graph. If $\alpha+2$ is even,
we will also show $U=V_{\alpha+2}$.
But if $\alpha+2$ is odd, then $U\psub V_{\alpha+2}$.}
The results here also demonstrate that, although $j:V_{\alpha+2}\to 
V_{\alpha+2}$ is incompatible with $\AC$,
the existence of such embeddings does actually \emph{imply} certain weaker 
choice principles (see Remark \ref{rem:choice_princ}).\footnote{This is 
analogous to the fact that the Axiom 
of Determinacy, while inconsistent with $\AC$, also implies certain weak choice 
principles.}

In \S\ref{sec:limit_Sigma_1_elementarity} we prove some more 
general results
in the limit case; in particular:

\begin{tm*}[\ref{tm:no_def_Sigma_1_V_delta_to_V_delta},
\ref{tm:j:V_eta_to_V_delta}] Let $\beta\leq\delta$ be limit ordinals
and $j:V_\beta\to V_\delta$ be 
$\Sigma_1$-elementary
and $\in$-cofinal, and suppose  that either $\beta=\delta$, or $j$ is fully 
elementary.
Then $j$ is not 
definable over $V_\delta$ from parameters.
\end{tm*}

Note that the $\beta<\delta$ case of this theorem applies to embeddings
which are compatible with $\AC$, in fact just around the level of extendible 
cardinals.

Finally, in \S\ref{sec:large_enough}, we discuss an old observation:
 if there is a Reinhardt cardinal, then there is an ordinal $\lambda$
such that for \emph{every} $\alpha\geq\lambda$, there is an elementary 
$j:V_\alpha\to V_\alpha$. So above $\lambda$, Theorem 
\ref{tm:cumulative_periodicity} applies, showing that cumulative hierarchy
(and correspondingly, the power set operation) is eventually periodic in nature.

Sections \S\ref{subsec:terminology} 
and \ref{sec:suzuki} cover background material.

We note some history on the development of the work. The results on the 
limit case in 
\S\ref{LimitSection} and \S\ref{sec:limit_Sigma_1_elementarity} are due to the 
second author,
and most of that material appeared in the informal notes
\cite{reinhardt_non-definability} (part 
\ref{item:Sigma_1_elem_j_not_in_own_Ult} of Theorem 
\ref{tm:no_def_Sigma_1_V_delta_to_V_delta},
and Theorem \ref{tm:j:V_eta_to_V_delta}, came later).
The analysis of embeddings $j:V_{\lambda+n}\to 
V_{\lambda+n}$ for limit $\lambda$ and $n=2$ in terms of Reinhardt ultrafilters,
in \S\ref{sec:reinhardt_filters}, was 
discovered in some form by the first author in 2017, and he 
communicated this to the second author shortly after the release of 
\cite{reinhardt_non-definability}. The first author then  
discovered Theorem \ref{tm:cumulative_periodicity}, 
and used this to generalize 
Woodin's $I_0$-theory
to higher levels (see
\cite{goldberg_even_numbers}).
A few months later, also attempting to generalize the first author's 
analysis of embeddings to $n>2$, the second 
author rediscovered Theorem \ref{tm:cumulative_periodicity}.
Our two proofs of non-definability in the even successor case
(Theorem \ref{EvenUndefinable}) were
different;
the one we give here is that due to the second author. The original one,
due to the
first author, can be seen in \cite{goldberg_even_numbers}.

\subsection{Terminology, notation, basic facts}\label{subsec:terminology}

We will assume the reader is 
familiar with basic first-order logic and set theory.
But much of the material, particularly in the earlier parts of the paper,
does not require extensive
background in set theory, so we aim to make 
at least those parts fairly broadly 
accessible.
Therefore  we do explain some points in the paper which are  standard, and 
summarize in this section some basic facts for convenience;
the reader should refer to texts like \cite{kunen_set_theory_2011}
for more details.

The language of set theory is the first-order language
with the membership relation $\in$. The Zermelo-Fr\"ankel
axioms are denoted by $\ZF$, and $\ZFC$ denotes $\ZF+\AC$,
where $\AC$ is the Axiom of Choice. We sometimes
discuss $\ZF(\dot{A})$, where $\dot{A}$ is an extra predicate
symbol; this is just like $\ZF$, but in the expanded language
with both $\in$ and $\dot{A}$, and incorporates the
Collection and Separation schemata
for all formulas in the expanded language.
A model of $\ZF(\dot{A})$ has the form $(V,\in,A)$,
abbreviated $(V,A)$, where $V$ is the universe of sets
and $A\sub V$. Thus, $A$ is automatically a class of this model
(and in the interesting case, $A$ is not already definable from parameters
over $V$).

We write $\Sigma_0=\Pi_0=\Delta_0$ for the class
of formulas (in the language of set theory)
in which all quantifiers are \emph{bounded},
meaning of the form ``$\all x\in y$'' or ``$\exists x\in y$''.
Then $\Sigma_{n+1}$ formulas are those of the form
``$\exists x_1,\ldots,x_n\psi(x_1,\ldots,x_n,\vec{y})$''
where $\psi$ is $\Pi_n$,
and $\Pi_{n+1}$ formulas are negations of $\Sigma_{n+1}$.
A relation is $\Delta_{n+1}$ if expressed by both
$\Sigma_{n+1}$ and $\Pi_{n+1}$ formulas.

Given structures $M=(\univ{M},R_1,R_2,\ldots,R_n)$
and $N=(\univ{N},S_1,S_2,\ldots,S_n)$ for the same first order language 
$\mathscr{L}$,
with universes $\univ{M}$ and $\univ{N}$ respectively,
a map $\pi:M\to N$ (literally, $\pi:\univ{M}\to\univ{N}$) is \emph{elementary},
just in case
\begin{equation}\label{eqn:elem_emb} M\sats\varphi(\vec{x})\iff 
N\sats\varphi(\pi(\vec{x})) \end{equation}
for all first order formulas $\varphi$ of $\mathscr{L}$ and all finite tuples
$\vec{x}\in M^{<\om}$. We can refine this notion
by considering formulas of only a certain complexity:
We say $\pi$ is \emph{$\Sigma_n$-elementary}
iff line (\ref{eqn:elem_emb}) holds
for all $\vec{x}\in M^{<\om}$ and $\Sigma_n$ formulas $\varphi$.

An elementary substructure is of course
the special case of this in which $\pi$ is just the inclusion map.
We write $M\elem N$ for a fully elementary substructure,
and $M\elem_n N$ for $\Sigma_n$-elementary.

Given $X\sub M$ and $p\in M$, $X$ is \emph{definable over $M$ from the 
parameter $p$}
iff there is a formula $\varphi\in\mathscr{L}$  such that for all $x\in M$
(literally $x\in\univ{M}$), we have
\[ x\in X\iff M\sats\varphi(x,p).\]
This can also be refined to \emph{$\Sigma_n$-definable from $p$},
if we demand $\varphi$ be a $\Sigma_n$ formula,
and likewise for $\Pi_n$. We say that $X$ is \emph{definable
over $M$ without parameters} if we can take $p=\emptyset$.
We say $X$ is \emph{definable over $M$ from parameters}
if $X$ is definable over $M$ from some $p\in M$.

Recall that a set $M$ is
\begin{enumerate}[label=--]
 \item \emph{transitive} iff $\all x\in M\all y\in x[y\in M]$,
 \item \emph{extensional} iff $\all x,y\in M[x\neq y\implies\exists z\in M[z\in 
x\iff z\notin y]]$;
\end{enumerate}
note these notions are $\Delta_0$. The \emph{Mostowski
collapsing Theorem} asserts that if $M$ is a set
and $E$ a binary relation on $M$ which is wellfounded
and $(M,E)$ satisfies \emph{$E$-extensionality}
(that is, $\all x,y\in M[x\neq y\implies\exists z\in M[zEx\iff \neg 
zEy]]$), then there is a unique transitive set $\bar{M}$,
and unique map $\pi:\bar{M}\to M$,
such that $\pi$ is an isomorphism
\[ \pi:(\bar{M},\in)\to(M,E); \]
here $\bar{M}$ is called the \emph{Mostowski}
or \emph{transitive collapse} of $(M,E)$,
and $\pi$ the \emph{Mostowski uncollapse map}.
The most important example of transitive sets in this paper
are the segments $V_\alpha$ of the cumulative hierarchy.

A key fact for transitive sets is that of \emph{absoluteness}
with respect to $\Delta_0$ truth:
Let $M$ be transitive. Then $\Delta_0$ formulas are \emph{absolute}
to $M$, meaning that if $\psi$ is $\Delta_0$ and $\vec{x}\in M^{<\om}$, then
\[ \psi(\vec{x})\iff [M\sats\psi(\vec{x})]. \]
Here the blanket assertion ``$\psi(\vec{x})$'' on the left
implicitly means ``$V\sats\psi(\vec{x})$'' where $V$ is the ambient universe
in which we are working. This equivalence is proven by an induction on the
formula length. It follows that if $\psi$ is $\Delta_0$ then
\[ [M\sats\exists y\psi(y,\vec{x})]\implies [\exists y\psi(y,\vec{x})],\]
(in fact any witness $y\in M$ also works in $V$), so conversely,
\[ [\all y\psi(y,\vec{x})]\implies [M\sats\all y\psi(y,\vec{x})].\]

We write $\OR$ for the class of all ordinals.
Ordinals $\alpha,\beta$ are represented as sets in the standard
form: $0=\emptyset$, $\alpha+1=\alpha\cup\{\alpha\}$,
and we take unions at limit ordinals $\eta$.
The standard ordering on the ordinals is then
$\alpha<\beta\iff\alpha\in\beta$, and this ordering
is wellfounded.
Being an ordinal is a $\Delta_0$-definable property,
because
$x$ is an ordinal iff $x$ is transitive and (the elements of)
$x$ are linearly ordered by $\in$. Therefore being an ordinal
is absolute for transitive sets, and preserved by $\Sigma_0$-elementary
embeddings between transitive sets.
That is, if $M,N$ are transitive and $x\in M$ then
\[ x\text{ is an ordinal }\iff M\sats x\text{ is an ordinal}, \]
and if $j:M\to N$ is also $\Sigma_0$ elementary then
\[ M\sats x\text{ is an ordinal }\iff N\sats j(x)\text{ is an ordinal}.\]
So this will hold in particular for the elementary embeddings $j:V_\alpha\to 
V_\beta$
that we consider. Note that transitivity of sets is also a $\Delta_0$-definable 
property, so absolute. Note that if $M$ is a transitive
set then $\OR\inter M$ is also an ordinal, in fact the least ordinal not in $M$.

If $N$ is a model of $\ZF$ (possibly non-transitive), we write
\[ \OR^N=\{\alpha\in N: N\sats\text{``}\alpha\in\OR\text{''}\}.\]
Similarly, if $\alpha\in\OR^N$ we write
\[ V_\alpha^N=\text{ the unique }v\in N\text{ such that 
}N\sats\text{``}v=V_\alpha\text{''}.\]
We use analogous superscript-$N$ notation whenever
we have a notion defined using some theory $T$
and $N\sats T$. So superscript-$N$ means ``as computed/defined in/over $N$''.

Given a set $x$, the \emph{rank} of $x$, denoted $\rank(x)$, denotes the least 
ordinal
$\alpha$ such that $x\sub V_\alpha$. (The Axiom of Foundation
ensures that this is well-defined.)

Given a function $f:X\to Y$, $\dom(f)$ denotes the domain of $f$,
$\rg(f)$  the range, and given $A\sub X$,
$f[A]$ or $f``A$ denotes the pointwise image of $A$.

Let $j:V\to M$ be elementary, where $M\sub V$ and $j$ is non-identity.
An argument by contradiction can be used to show that there is
an ordinal $\kappa$ such that $j(\kappa)>\kappa$,
and the least such is called the \emph{critical point} of $j$,
denoted $\crit(j)$. The same holds more generally,
for example if $j:M\to M$ is elementary where $M$ is a transitive
set or class.
It follows that, in particular, such a $j$ cannot be surjective,
so there is no non-trivial
$\in$-isomorphism of transitive $M$.

If $j:M\to N$ is elementary between
transitive $M,N$, then $M\iso\rg(j)\elem N$,
and 
$\rg(j)$
is a wellfounded extensional set, and therefore the Mostowski
collapsing theorem applies to it. The transitive collapse
is just $M$, and $j$ is the uncollapse map.
So from $j$ we can compute $\rg(j)$ (and $M=\dom(j)$),
and from $\rg(j)$ we can recover $M,j$.

\section{Suzuki's Fact: Non-definability of $j:V\to V$}\label{sec:suzuki}

Suzuki \cite{suzuki_no_def_j} proved the following basic fact.
We will use variants of its proof later, and the 
proof is short,
so for expository purposes, we include it as a warm-up. Everything in this 
section
is well known.
\begin{fact}[Suzuki]\label{fact:suzuki_no_def_j}
Assume $\ZF$.\footnote{That is, we are assuming
that the universe $V\sats\ZF$. We often use this language
and then make statements which are to be interpreted in/over $V$.} Then no 
class $k$ which is definable from parameters is a non-trivial 
elementary embedding
 $k:V\to V$.
\end{fact}

Here when we say simply ``definable from parameters'',
we mean over $V$. Of course, the 
theorem is really a theorem scheme,
giving one statement for each possible formula $\varphi$
being used to define $k$ (from a parameter).
In order to give the proof, we need a couple of lemmas.
The  first is a little easier to consider
in the case that $\alpha$ in the proof is a limit ordinal,
but the proof goes through in general.

\begin{lem}\label{lem:Sigma_1-elem_pres_rank}
 Let $j:V_\delta\to V_\lambda$ be $\Sigma_1$-elementary.
 Then $j(V_\alpha)=V_{j(\alpha)}$ for all $\alpha<\delta$.
\end{lem}
\begin{proof}
Fix $\alpha<\delta$.
Note that $V_\delta$ satisfies the following statements about the parameters 
$\alpha$ and $V_\alpha$:\footnote{When we write ``$V_\alpha$'' in the 
3 statements,
we refer to the object $x=V_\alpha$ as a parameter,
as opposed to the object defined as the $\alpha$th stage
of the cumulative hierarchy. But note that 
the ``$\beta$'' and ``$V_\beta$'' are quantified variables,
and here $V_\beta$ \emph{does} refer to the $\beta$th
stage of the cumulative hierarchy.}
\begin{enumerate}[label=--]
\item ``$V_\alpha$ is transitive''
 \item ``For every $X\in V_\alpha$
and every $Y\sub X$, we have $Y\in V_\alpha$'',
\item ``$V_\alpha$ satisfies `For every ordinal $\beta$, $V_\beta$ 
exists'.''.\footnote{The reader might notice that this needs to be formulated 
appropriately, because if $\alpha=\beta+1$, then the standard definition
of $\left<V_\gamma\right>_{\gamma\leq\beta}$ is the function
$f:\beta+1\to V$ where $f(\gamma)=V_\gamma$, and $f\notin V_\alpha$.
But it is straightforward to reformulate things appropriately.
For the case in which $j:V_\delta\to V_\delta$ and $\delta$ is a limit,
one can also get around these things in other ways,
since we can just talk about elements of $V_\delta$, instead of
literally talking about something that $V_\alpha$ satisfies.}
\end{enumerate}
The first statement here is $\Sigma_0$
(in parameter $V_\alpha$),
the second is $\Pi_1$,
and the third $\Delta_1$, so $V_\lambda$ satisfies
the same assertions of the parameter $j(V_\alpha)$.
It follows that $j(V_\alpha)=V_\beta$ for some $\beta<\lambda$. But also 
$\alpha=V_\alpha\inter\OR$, another fact preserved by $j$
(again by $\Sigma_1$-elementarity),
so $j(\alpha)=j(V_\alpha)\inter\OR$, so $\beta=j(\alpha)$. 
\end{proof}

The following fact is  \cite[Proposition 5.1]{kanamori}
(though it is stated under the assumption that $M,N$
are transitive proper class inner models there).
We will need to prove analogues later (and we will
actually appeal to a relativization of it to models of $\ZF(\dot{A})$,
the proof of which we leave it to the reader to fill in). So let 
us look at the proof, also as a warm-up:
\begin{fact}\label{fact:ZF_Sigma_1-elem_cofinal_implies_full_elem}
 Let $M,N$ be models of $\ZF$.
 Let $j:M\to N$ be 
$\Sigma_1$-elementary and $\in$-cofinal. Then
 $j$ is fully elementary.
\end{fact}
\begin{proof}
 We prove by induction on $n<\om$, that $j$ is $\Sigma_n$-elementary.
 
 Because $j$ is $\Sigma_1$-elementary,
 we have $j(V_\alpha)=V_{j(\alpha)}$ for each $\alpha$.
 
 Suppose $j$ is $\Sigma_n$-elementary where $n\geq 1$.
 Let $C_n\sub\OR^M$ be the $M$-class of all $\alpha$
 such that $V^M_\alpha\elem_n V^M_\lambda$  (note that $C_n$ is as defined
 over 
$M$,
 without parameters). $\ZF$ proves (via standard
 model theoretic methods) that $C_n$ is unbounded in $\OR$.
 
 Let $\alpha\in C_n$.
We claim that $j(\alpha)\in C^N_n$ (with $C_n^N$ defined analogously over 
$N$; see \S\ref{subsec:terminology}). For suppose
$N\sats\varphi(x)$ where $x\in V_{j(\alpha)}^N$
and $\varphi$ is $\Sigma_n$, but that $V_{j(\alpha)}^N\sats\neg\varphi(x)$.
The existence of such an $x$ is a $\Sigma_n$ assertion about the 
parameter
$V_{j(\alpha)}^N$, satisfied by $N$,
so $M$ satisfies the same about $V_\alpha^M$ (by $\Sigma_n$-elementarity of 
$j$). But $\alpha\in C_n$, contradiction.

Now suppose that $N\sats\varphi(j(x))$, where $\varphi$ is $\Sigma_{n+1}$.
Then by the $\in$-cofinality of $j$
and the previous remarks, we may pick $\alpha\in C_n$ such that
$x\in V_\alpha^M$ and $V_{j(\alpha)}^N\sats\varphi(j(x))$.
But then $V_\alpha^M\sats\varphi(x)$,
and since $\alpha\in C_n$, it follows that $M\sats\varphi(x)$, as desired.
\end{proof}

\begin{proof}[Proof of \ref{fact:suzuki_no_def_j}] Suppose that $k:V\to V$ is 
elementary
and there is a $\Sigma_n$ formula $\varphi$ and
$p\in V$ such that for all $x,y$, we have
\[ k(x)=y\iff \varphi(p,x,y).\]
Given any parameter $q$, attempt to define a function $j_q$ by:
\[ j_q(x)=y\iff\varphi(q,x,y).\]
Say that $q$ is \emph{bad} iff $j_q:V\to V$
is a $\Sigma_1$-elementary, non-identity map.
Because $j_q$ is defined
using the fixed formula $\varphi$ and we only demand $\Sigma_1$-elementarity,
\emph{badness} is a definable notion (without parameters).
And $p$ above is bad.

By Fact
\ref{fact:ZF_Sigma_1-elem_cofinal_implies_full_elem} above,
if $q$ is bad then $j_q$ is in fact fully elementary.

Now let $\kappa_0$ be the least critical point
$\crit(j_q)$ among all bad parameters $q$.
Note then that the singleton $\{\kappa_0\}$ is definable over $V$,
from no parameters. (So there is a formula $\psi$
such that $\psi(x)\iff x=\kappa_0$, for all sets $x$.)

Let $q_0$ witness the choice of $\kappa_0$.
As mentioned above, $j_{q_0}$ is in fact fully elementary,
and we have $\crit(j_{q_0})=\kappa_0$.
So $j_{q_0}(\kappa_0)>\kappa_0$,
whereas $j_{q_0}(\alpha)=\alpha$ for all $\alpha<\kappa_0$.
Since $j_{q_0}$ is order-preserving, $\kappa_0\notin\rg(j_{q_0})$.
But by the (full) elementarity of $j_{q_0}:V\to V$
and definability of $\{\kappa_0\}$,
we must have $j_{q_0}(\kappa_0)=\kappa_0\in\rg(j_{q_0})$,
a contradiction.
\end{proof}

We remark that Suzuki actually proved a more general theorem,
considering elementary embeddings of the form $j:M\to V$
where $M\sub V$, and $j$ is definable from parameters.

\section{Definability of rank-to-rank embeddings}\label{sec:def_rank-to-rank}

\subsection{The limit case}\label{LimitSection}
Most investigations of rank-to-rank embeddings to date have focused on 
elementary embeddings \(j : V_\alpha\to V_\alpha\) where
\(\alpha = \kappa_\omega(j)\) or \(\alpha = \kappa_\omega(j) + 1\), 
since assuming Choice, these are the only 
rank-to-rank embeddings there are. The following very simple 
fact turns out to 
play a central role in these investigations: if \(\lambda\) is a limit ordinal, 
an elementary embedding from \(V_\lambda\) to \(V_\lambda\) extends in at most 
one way to an elementary embedding from \(V_{\lambda+1}\) to \(V_{\lambda+1}\).

\begin{dfn}
 For a structure $M$, $\mathscr{E}(M)$ denotes the set of all elementary
 embeddings $j:M\to M$.
\end{dfn}

\begin{dfn}\label{CanonicalExtension}
Let \(\lambda\) be a limit ordinal and $j\in \mathscr{E}(V_\lambda)$.
The {\it canonical extension of \(j\)} is the function
$j^+ : V_{\lambda+1}\to V_{\lambda+1}$
defined 
$j^+(X) = \bigcup_{\alpha < \lambda} j(X\cap V_\alpha)$.
\end{dfn}
The canonical extension $j^+$ is a function $V_{\lambda+1}\to 
V_{\lambda+1}$.
However, it is well known that it can fail to be elementary. (For 
example, let $\kappa$ be least
such that there is an elementary $j:V_\lambda\to V_\lambda$
with $\crit(j)=\kappa$,
and show that $j^+$ is not elementary.)
But if \(j\) 
{\it does} extend to some
\(i\in\mathscr{E}(V_{\lambda+1})\), or even just to a $\Sigma_1$-elementary 
$i:V_{\lambda+1}\to V_{\lambda+1}$, then clearly 
\(i(V_\lambda)=V_\lambda\) and \(i = j^+\).

Let $\lambda$ be a limit ordinal.
It follows that every \(j\in\mathscr{E}(V_{\lambda+1})\)
is definable over \(V_{\lambda+1}\) from 
parameters, in fact, from its own restriction \(j\rest 
V_\lambda\). (Since \(V_\lambda\) is closed under 
ordered 
pairs, \(j\restriction V_\lambda\in
V_{\lambda+1}\).) However,
$j$ is not definable over $V_{\lambda+1}$
from any element of $V_\lambda$, and no
\(j\in\Eee(V_\lambda)\) is definable from parameters over 
$V_\lambda$:
\begin{tm}\label{LimitUndefinable}
 Let $\delta$ be an ordinal, $j\in\Eee(V_\delta)$
and  $p\in V_\delta$,
with $j$  definable over $V_\delta$ from the parameter $p$.
 Then $\delta=\beta+1$ is a successor and $\rank(p)=\beta$.
\end{tm}
\begin{proof}
Suppose not. We adapt the proof of Suzuki's Fact.
Fix $(k,\varphi,\beta)$ such that  $k<\om$, $\varphi$ is a $\Sigma_k$ 
formula and
 $\beta<\delta$, and
for some $p\in V_\beta$ we have
 $j_p\in\Eee(V_\delta)$ where
 \[ j_p=\{(x,y)\in V_\delta\cross V_\delta: 
V_\delta\sats\varphi(p,x,y)\}.\]
Say that $q\in V_\beta$ is \emph{$\om$-bad} if $j_q\in\Eee(V_\delta)$.

Let $\mu_0$ be the least critical point
among all such (fully elementary) embeddings $j_q$ (minimizing
over all $\om$-bad parameters $q$).
Let $p_0\in V_\beta$ witness this, so 
$j_{p_0}\in\mathscr{E}(V_\delta)$ and $\crit(j_{p_0})=\mu_0$.

For $n<\om$, say that $q\in V_\beta$ is
\emph{$n$-bad} iff
$j_q:V_\delta\to V_\delta\text{ and is }\Sigma_n\text{-elementary}$.
Let
$A_n=\{q\in V_\beta: q\text{ is }n\text{-bad}\}$.
So $A_n\in V_\delta$ and note that $A_n$ is definable over $V_\delta$ from the 
parameter $\beta$.

Since $j=j_{p_0}$ is fully elementary, $j(A_n)\inter V_\beta=A_n$
(note $j(\beta)\geq\beta$).
Let $A=\bigcap_{n<\om}A_n$, so $A\in V_\delta$.
Note that $j_q\in\Eee(V_\delta)$ for every $q\in A$.

The sequence $\left<A_n\right>_{n<\om}$ can easily be coded
by a set in $V_\delta$ (with methods like in the next section; if $\delta$ is a 
limit then it is in fact literally in $V_\delta$), 
and therefore
\[ j(A)=\bigcap_{n<\om}j(A_n), \]
so $p_0\in j(A)$.
Therefore $V_\delta\sats$``$\exists q\in j(A)$
such that $\crit(j_q)<j(\mu_0)$'' (as witnessed by $p_0$).
Pulling this back with the elementarity of $j$,
$V_\delta\sats$``$\exists q\in A$ such that
$\crit(j_q)<\mu_0$.'' But this contradicts the minimality of $\mu_0$.
\end{proof}

\subsection{A flat pairing function}\label{subsec:coding}
If $\delta$ is a limit ordinal then $V_\delta$ is closed under pairs $\{x,y\}$,
and hence, ordered pairs $(x,y)$, represented in the standard fashion as 
$(x,y)=\{\{x\},\{x,y\}\}$.
But this fails in the successor case, at least when we use this standard 
representation: For example, $V_\delta\in V_{\delta+1}$ but the pair 
$\{V_\delta,\emptyset\}\notin V_{\delta+1}$. 
It is therefore useful to employ a different representation or \emph{coding} of 
ordered pairs with the property that for every infinite ordinal $\alpha$, for all $x,y\in 
V_\alpha$, the code \(\lceil x,y\rceil\) for the pair $(x,y)$ is an element
of $V_\alpha$. In this case, the function
\((x,y)\mapsto \lceil x,y\rceil\) is called a {\it flat pairing function}.

There are many different flat pairing functions,
and which one we use will not really be relevant in our applications.
All we will really require of the pairing function is that it be a \(\Sigma_0\)-definable
injection \(\Phi : V\times V\to V\) such that \(\Phi``(V_\alpha\times V_\alpha)\subseteq V_\alpha\)
for all infinite ordinals \(\alpha\).

Nevertheless, let us define the {\it Quine-Rosser pairing function}, which
is officially the pairing function we employ below.
The basic idea is to code a pair \((x,y)\) by a
labeled disjoint union of \(x\) and \(y\).
Somewhat more precisely, we will take two disjoint copies
$V_0$ and $V_1$ of the universe $V$
and  
bijections $f_0 : V\to V_0$ and $f_1: V\to V_1$,
which are both rank-preserving over all sets of rank $\geq\om$.
The ordered pair $(x,y)$ is then coded by the set 
$\lceil x,y\rceil = f_0``x\cup f_1``y$.

To implement this idea without leaving \(V\),
let $V_0$ be the class of sets that do not contain the empty set
and let $V_1$ be the class of sets that do.
Let $s : V\to V$ be defined by setting
\(s(n) = n+1\) for all \(n < \omega\) and \(s(u) = u\) for all \(u\notin 
\omega\). Then let $f_0: V\to V_0$ be defined by $f_0(X)=s``X$
and  \(f_1 : V\to V_1\) be defined by
\(f_1(u) = (s``X)\cup\{\emptyset\}$.

\begin{dfn}
  For sets $x,y\in V$, let 
\[\lceil x,y\rceil = f_0``x\cup f_1``y,\]
where \(f_0\) and \(f_1\) are as defined above.
The set $\lceil x,y\rceil$ is the {\it Quine-Rosser pair}
coding $(x,y)$.
\end{dfn} 
The Quine-Rosser pairing function \((x,y)\mapsto \lceil x,y\rceil\)
establishes a bijection from \(V\times V\) to \(V\) whose inverse is the function
\(z\mapsto (f_0^{-1}[z],f_1^{-1}[z])\).
It is easy to show that for any set \(u\),
\(\rank(f_0(u))\) and \(\rank(f_1(u))\)
are bounded by \(1+\text{rank}(u)\), which implies
\[\rank(\lceil x,y\rceil ) \leq 1 + \max \{\rank(x),\rank(y)\}.\]
In particular, for any infinite ordinal \(\alpha\),
the Quine-Rosser pairing function restricts to a bijection from 
\(V_{\alpha}\times V_\alpha\)
to \(V_\alpha\).
Moreover, this function is
$\Sigma_0$-definable over the structure $(V_\alpha,\in)$.\footnote{
Of course we can't presuppose another notion of pairing
when we discuss the definability. So to be clear,
the definability means
there is a $\Sigma_0$ formula $\varphi$ of 3 variables such
that $\lceil x,y\rceil =z\iff \varphi(x,y,z)$.}

From now on, we shift notation, and whenever we talk about ordered pairs,
we mean Quine-Rosser pairs, and whenever we talk about
binary relations $R$ on $V_\alpha$ (where $\alpha\geq\om$)
we will literally mean that $R$ is a set of Quine-Rosser pairs,
and similarly for $n$-ary relations. Therefore 
 $R\in V_{\alpha+1}$. Moreover, note that there is a $\Sigma_0$ formula
in the language of set theory such that for any such
$\alpha$ and binary relation $R$ on $V_\alpha$ and $x,y\in V_\alpha$,
we have $xRy$ iff $V_{\alpha+1}\sats\varphi(R,x,y)$. This will be used
in particular for functions $f:V_\alpha\to V_\alpha$.

\subsection{The successor case}
Our observations so far suggest the following natural 
questions. Suppose \(\xi\) is an ordinal and
$j\in\Eee(V_{\xi+2})$.
Can $j$ be definable over $V_{\xi+2}$ from some parameter
(necessarily of rank $\xi+1$)?
And if $n\geq 1$, then more specifically, can $j$ be definable over 
$V_{\xi+2}$ from $j\rest V_{\xi+1}$? Note here that, because
we are using Quine-Rosser pairs, $j\rest V_{\xi+1}\in V_{\xi+2}$.
In this section, we answer these questions.
It turns out that most of the results from the limit case generalize
to the case of arbitrary even ordinals.

At first glance, it seems that the definition
of the canonical extension operation ({Definition \ref{CanonicalExtension}})
makes fundamental use of the assumption that \(\lambda\) is a limit ordinal.
In particular, this definition exploits the
hierarchy \(\langle V_\alpha\rangle_{\alpha < \lambda}\) stratifying
\(V_{\lambda}\); 
this hierarchy seems to have no analog at the successor even levels.
But on further thought, we could have defined \(j^+(X)\) for \(X\in V_{\lambda+1}\)
as follows: 
\[j^+(X) = \bigcup \{j(a) : a\in V_\lambda\text{ and }a\subseteq X\}\]
Thus \(j^+(X)\) is the union of the image of \(j\) on all the subsets of \(X\)
that belong to \(V_\lambda\).

At successor ordinals we must generalize this slightly, instead
taking the union of the image of \(j\) on all the subsets of \(X\) 
that are {\it coded} in \(V_\lambda\).
\begin{dfn}\label{dfn:coding}
  Suppose \(a\) and \(b\) are sets.
  For any set \(x\),
  let \((a)_x\) 
  denote the set \(\{y : \lceil x,y\rceil \in a\}\),
  and let \((a \mid b) = \{(a)_x : x\in b\}\).
\end{dfn}
Thus for \(a,b\in V_\lambda\), \((a \mid b)\) is the subset of \(V_{\lambda}\)
whose elements are the sections \((a)_x\in V_\lambda\) of the 
binary relation coded by \(a\)
that are indexed by some \(x\in b\). 
Say a set \(X\subseteq V_\lambda\) is {\it coded in \(V_\lambda\)}
if \(X = (a \mid b)\) for some \(a,b\in V_\lambda\).
For \(\lambda\) a limit ordinal, every set coded in \(V_\lambda\) belongs to \(V_\lambda\),
but if \(\lambda\) is a successor ordinal, then the sets coded in \(V_\lambda\)
are precisely those \(X\subseteq V_\lambda\) such that
there is a partial surjection from \(V_{\lambda-1}\) onto \(X\). 
\begin{dfn}
Suppose \(\lambda\) is an ordinal. For any function \(j : V_\lambda \to V_\lambda\), 
the {\it canonical extension of \(j\)} is the function \(j^+ : V_{\lambda+1}\to V_{\lambda+1}\)
defined by 
\[j^+(X) = \bigcup \{(j(a) \mid j(b)) : a,b\in V_{\lambda}\text{ and }(a \mid b)\subseteq X\}
\qedhere\]
\end{dfn}
While \(j^+\) is well-defined for any function \(j\), it is not of much interest
unless \(j\) has the property that \((j(a) \mid j(b)) = (j(a') \mid j(b'))\)
whenever \((a \mid b) = (a' \mid b')\).

Suppose \(a,b\in V_\lambda\), \(X\in V_{\lambda+1}\), 
and \((a \mid b)\subseteq X\).
The fact that \((a \mid b)\) is included in \(X\) is a first-order expressible property of \(a,b,\) and \(X\) 
in \(V_{\lambda+1}\), so
for any \(k \in \mathscr E(V_{\lambda+1})\),
\((k(a) \mid k(b))\subseteq k(X)\).
It follows that \((k\restriction V_\lambda)^+(X)\subseteq k(X)\), whether \(\lambda\) is even or odd.
The reverse inclusion, however, will be true if and only if \(\lambda\) is even.
\begin{dfn}
  Suppose \(\lambda\) is an ordinal. An embedding \(j : V_\lambda\to V_\lambda\) is
  {\it cofinal} if for any set \(c\in V_\lambda\), there exist sets \(a,b\in V_\lambda\) such that
  \(c\in (j(a) \mid j(b))\).
\end{dfn}
Equivalently, \(j : V_\lambda\to V_\lambda\) is cofinal if \(j^+(V_\lambda) = V_\lambda\).
It follows immediately that if \(k\in \mathscr E(V_{\lambda+1})\) and
\(k = (k\restriction V_\lambda)^+\), 
then \(k\restriction V_\lambda\) must be cofinal. The converse is also true:

\begin{lem}\label{CofinalEmbeddingsExtendCanonically}
  Suppose \(k\in \mathscr E(V_{\lambda+1})\)
  and \(k\restriction V_\lambda\) is cofinal. Then \(k = (k\restriction V_\lambda)^+\).
  \begin{proof}
    Fix \(X\in V_{\lambda+1}\). Our comments above show that
    \((k\restriction V_\lambda)^+(X)\subseteq k(X)\). For the reverse
    inclusion, fix \(c\in k(X)\). We will show \(c\in (k\restriction V_\lambda)^+(X)\).

    Since \(k\restriction V_\lambda\) is cofinal,
    there are sets \(a,b\in V_\lambda\) such that \(c\in (k(a) : k(b))\).
    Let \(b' = \{x\in b : (a)_x\in X\}\), so that \((a \mid b') = (a \mid b)\cap X\).
    Now  \[c\in  (k(a)  \mid  k(b))\cap k(X) = 
    k((a  \mid  b)\cap X) = k\left(a  \mid  b'\right)\subseteq (k\restriction V_\lambda)^+(X)\]
    This shows \(k(X)\subseteq (k\restriction V_\lambda)^+(X)\), completing the proof.
  \end{proof}
\end{lem}

The periodicity phenomenon is driven by the following lemma:
\begin{lem}\label{AutoCofinal}
  Suppose \(j : V_{\lambda+2}\to V_{\lambda+2}\) is an elementary embedding
  such that \((j\restriction V_\lambda)^+ = j\restriction V_{\lambda+1}\).
  Then \(j\) is cofinal.
  \begin{proof}
    Fix \(j : V_{\lambda+2}\to V_{\lambda+2}\) and \(C\in V_{\lambda+2}\).
    We must show that there exist \(A,B\in V_{\lambda+2}\) such that
    \(C\in (j(A)  \mid  j(B))\). Let \(B\) consist of those sets \(x\in V_{\lambda+1}\)
    such that the binary relation \(\{(a,b) : \lceil a,b\rceil\in x\}\) coded by \(x\)
    is the graph of a function \(f_x : V_{\lambda}\to V_{\lambda}\). 
    By elementarity, \(j(B) = B\).

    Now define
    \(A = \{\lceil x,y\rceil \in B\times V_{\lambda+1} : f_x^+(y)\in C\}\).
    In other words, for each \(x\in B\), \((A)_x = (f_x^+)^{-1}[C]\).
    Now \[j(A) = \{\lceil x,y\rceil \in B\times V_{\lambda+1} : f_x^+(y)\in j(C)\}.\]
    Let \(w = \{\lceil a, j(a)\rceil: a\in V_\lambda\}\), so that 
    \(f_w = j\restriction V_\lambda\). Then \(w\in B\) and
    \[(j(A))_w = (f_w^{+})^{-1}(j(C)) = ((j\restriction V_\lambda)^+)^{-1}[j(C)] = 
    (j\restriction V_{\lambda+1})^{-1}[j(C)] = C.\]
    Therefore \(C= (j(A))_w \in (j(A)  \mid  B) = (j(A)  \mid  j(B))\), as desired.
  \end{proof}
\end{lem}

\begin{tm}\label{LocalPeriodicity}
  Suppose \(\lambda\) is an even ordinal and \(j : V_\lambda \to V_\lambda\) is an
  elementary embedding. Then \(j\) is cofinal. Suppose in addition that
  \(j\) extends to an elementary embedding \(k : V_{\lambda+1}\to V_{\lambda+1}\).
  Then \(k = j^+\).
  \begin{proof}
    Assume by induction that the corollary is true for ordinals smaller than \(\lambda\).
    If \(\lambda\) is a limit ordinal, then \(j\) is trivially cofinal.
    If \(\lambda\) is a successor ordinal, then by our induction hypothesis applied to \(\lambda-2\),
    \(j\restriction V_{\lambda+1} = (j\restriction V_\lambda)^+\), so
    \(j\) is cofinal by Lemma \ref{AutoCofinal}.
    Since \(j\) is cofinal, 
    Lemma \ref{CofinalEmbeddingsExtendCanonically} implies
    that if \(j\) extends to an elementary embedding \(k : V_{\lambda+1}\to V_{\lambda+1}\),
    then \(k = j^+\). This completes the proof.
  \end{proof}
\end{tm}

The requirement that \(\lambda\) be even in the previous theorem is unusual, 
but one can show that the theorem fails whenever \(\lambda\) is odd.
The proof given here is an elaboration on that of 
{Theorem \ref{LimitUndefinable}}.
\begin{tm}\label{EvenUndefinable}
Suppose \(\alpha\) is an ordinal,
\(j \in \mathscr E(V_\alpha)\),
and \(a,b\in V_{\alpha}\). Then \(j\) is not definable over
\(V_\alpha\) from parameters in \((j(a)  \mid  j(b))\).
\begin{proof}
Suppose towards a contradiction that the theorem fails.
Then there is
a formula \(\varphi(v_0,v_1,v_2)\) and a parameter \(p\in (j(a)  \mid  j(b))\) such that
\[j(u) = w\iff V_{\alpha}\sats \varphi(u,w,p)\] for all \(u,w\in 
V_{\alpha}\).
For \(q\in V_{\alpha}\), define a relation
\[j_q = \{(u,w)\in V_{\alpha}^2 :  V_{\alpha}\sats\varphi(u,w,q)\}.\]

For \(n \leq \omega\), \(q\in V_\alpha\) is {\it \(n\)-bad}
if $j_q : V_{\alpha}\to V_{\alpha}$ is a nontrivial \(\Sigma_n\)-elementary embedding
and there exist \(a',b'\in V_\alpha\) such that \(q \in (j_q(a')  \mid  j_q(b'))\).

So $p$ is $\om$-bad. Let $\kappa = \min \{\crit(j_q) :  q\text{ is 
}\omega\text{-bad}\}$.
Fix an \(\omega\)-bad parameter \(r\) such that \(\crit(j_r) = \kappa\).

Fix \(c,d\in V_\alpha\) with \(r \in (j_r(c) \mid j_r(d))\).
For each \(n \leq \omega\), let \[d_n = \{x\in d :  (c)_x\text{ is 
\(n\)-bad}\}\]
By the elementarity of \(j_r\),
\[j_r(d_n) = \{x\in j_r(d) :  (j_r(c))_x\text{ is }n\text{-bad}\}.\]
Let \(e = \{\lceil n,x\rceil : x\in d_n\}\), so that \((e)_n = d_n\).
Since \(d_\omega = \bigcap_{n < \omega} (e)_n\),
\(j_r(d_\omega) = \bigcap_{n < \omega} (j_r(e))_n = \bigcap_{n <\omega} j_r(d_n)\).
It follows that 
\[j_r(d_\omega) = \bigcap \{x \in j_r(d) : j_r(c)_x \text{ is 
}\omega\text{-bad}\}.\]
In particular, \(r\in (j_r(c)  \mid  j_r(d_\omega))\) 
and every \(q\in (j_r(c)  \mid  j_r(d_\omega))\) is \(\omega\)-bad, so
\[\min \{\crit(j_q) : q\in (j_r(c)  \mid  j_r(d_\omega))\}= \kappa.\]
Therefore letting \(\bar \kappa = \min \{\crit(j_q) : q\in (c 
\mid d_\omega)\}\),
\(j_r(\bar \kappa) = \kappa\), which contradicts that \(\kappa\) is the
critical point of \(j_r\).
\end{proof}
\end{tm}

Putting everything together, we can
now prove Theorem \ref{tm:cumulative_periodicity};
that is, given 
$j\in\mathscr{E}(V_\lambda)$,
then $j$ is definable from parameters
over $V_\lambda$ iff $\lambda$ is odd:

\begin{proof}[Proof of Theorem \ref{tm:cumulative_periodicity}]
Suppose \(\lambda\) is even. Then by Theorem \ref{LocalPeriodicity},
\(j\) is cofinal, which means that every \(p\in V_\lambda\)
belongs to \((j(a):j(b))\) for some \(a,b\in V_\lambda\).
Therefore by Theorem \ref{EvenUndefinable},
\(j\) is not definable from any parameter in \(V_\lambda\).

On the other hand, if \(\lambda\) is odd, then \(j = (j\restriction V_{\lambda-1})^+\)
by Theorem \ref{LocalPeriodicity}, and therefore \(j\) is definable over \(V_\lambda\)
from \(j\restriction V_\lambda\), or more precisely from 
the set \(\{\lceil x,j(x)\rceil : x\in V_{\lambda-1}\}\), which belongs to \(V_\lambda\).
\end{proof}

\section{Reinhardt ultrafilters}\label{sec:reinhardt_filters}
Solovay's discovery of supercompactness in the late 1960's marked the beginning 
of the modern era of large cardinal theory.
In the context of ZFC, supercompactness has both a combinatorial 
characterization in terms of normal ultrafilters and a ``model theoretic'' 
characterization in terms of elementary embeddings \(j : V\to M\) where \(M\) 
is 
an inner model.
In the choiceless context, 
however, the equivalence between 
the usual characterizations 
is no longer 
provable, and instead supercompactness 
splinters into a number of inequivalent but interrelated concepts.

The 
rank-to-rank embeddings \(j : V_{\delta}\to V_{\delta}\) studied here 
exhibit features  reminiscent of
super\-com\-pact\-ness.
In this section we evidence this via a characterization in terms
of normal ultrafilters  in the case
that $\delta=\alpha+2$.
\footnote{
In this section we assume familiarity with ultrapowers as used in set theory;
the reader familiar with supercompactness measures should be fine.}
But since these embeddings force us into the choiceless
realm, we must deal with the subtleties this brings.
A key issue in this regard is that one needs to be more careful regarding
{\L}o\'{s}'s Theorem for ultrapowers, given the role of choice in its usual 
proof. 

\subsection{Ultrapowers and {\L}o\'{s}'s Theorem}\label{subsec:ultrapowers}

In this section we give a quick
review of some standard  
background. We assume familiarity with (ultra)filters, which can be found in 
standard texts. If 
$\mathscr{F}$ is a filter
over a set $X$ (so $X\in\mathscr{F}$)
and $\varphi$ is some property,
 say that $\varphi(x)$ holds for \emph{$\mathscr{F}$-almost all }$x$
(or just \emph{almost all} $x$)
iff $\{x\in X\bigm|\varphi(x)\}\in\mathscr{F}$.

We first recall the definition of \emph{ultrapowers} in our context.
Let $\gamma,\beta\in\OR$ and let $\mathscr{F}$ be any ultrafilter over 
$V_{\gamma}$. Let $\mathscr{U}$ denote the set of all functions 
$f:V_{\gamma}\to V_\beta$. We define a binary relation over $\mathscr{U}$ by
\[ f\approx_{\mathscr{F}}g \iff \{x\in 
V_\alpha\bigm|f(x)=g(x)\}\in\mathscr{F}.\]
Because $\mathscr{F}$ is a filter, it is easy to see that 
$\approx_{\mathscr{F}}$
is an equivalence relation; let $[f]^{V_\beta}_{\mathscr{F}}$ be the 
equivalence class of $f$, where we just write $[f]$ if there is no ambiguity. 
We 
define also the relation
\[ f\in_{\mathscr{F}}g\iff\{x\in V_\alpha\bigm|f(x)\in g(x)\}\in\mathscr{F}.\]
Then $\in_{\mathscr{F}}$ respects $\approx_{\mathscr{F}}$.
The \emph{ultrapower} $\Ult(V_\beta,\mathscr{F})$
of $V_\beta$ by $\mathscr{F}$ is the structure $(U,\in^U)$,
where $U=\{[f]\bigm|f\in\mathscr{U}\}$, and $\in^U$ is the binary relation on 
$U$
induced by $\in_{\mathscr{F}}$. The \emph{ultrapower embedding}
$i^{V_{\beta}}_{\mathscr{F}}:V_\beta\to U$ is defined
$i^{V_{\beta}}_{\mathscr{F}}(x)=[c_x]$ where $c_x\in\mathscr{U}$
is the constant function with constant value $x$.

Now let us say that \emph{$\Sigma_n$-{\L}o\'{s}' Theorem for $U$} holds
iff for all $\Sigma_n$ formulas
$\varphi$ (in the language of set theory) and functions 
$f_1,\ldots,f_n\in\mathscr{U}$, we have
\[\begin{array}{cl} &U\sats\varphi([f_1],\ldots,[f_n])\\
\iff& 
V_\beta\sats\varphi(f_1(x),\ldots,f_n(x))\text{ for almost 
all }x\in V_{\gamma}.\end{array}\]
We just say \emph{{\L}o\'{s}' theorem holds for $U$} if $\Sigma_n$-{\L}o\'{s}' 
theorem holds 
for all $n<\om$.
For atomic formulas $\varphi(u,v)$ (``$u=v$'' and ``$u\in v$'')
the stated equivalence holds by definition.
Assuming $\AC$ it holds for all formulas, as proved by induction on  formula 
complexity. The only step that uses $\AC$ is that for quantifiers:
suppose for example that
\[ \sigma\in\mathscr{F}\text{ and for all }x\in\sigma\text{ we have } 
V_\beta\sats\exists w\ \varphi(w,f(x)).\]
Then we want  
$U\sats\exists w\ \varphi(w,[f])$,
which needs some $w\in\mathscr{U}$ with
$U\sats\varphi([w],[f])$.
So we need $w:V_{\gamma}\to V_\beta$ and 
 by induction, we need some $\sigma'$ such that
\[ \sigma'\in\mathscr{F}\text{ and for all 
}x\in\sigma'\text{, we have }V_\beta\sats\varphi(w(x),f(x)).\]
Using $\AC$, we can in fact take $\sigma'=\sigma$
and $w$ to be an appropriate choice function.
But it is important here that we don't actually require $\sigma'=\sigma$;
so even if $\AC$ fails and there is no choice function with domain $\sigma$,
there might be one with a smaller domain $\sigma'\in\mathscr{F}$.

If {\L}o\'{s}' Theorem holds for $U$ then the ultrapower 
embedding $i:V_\beta\to U$ is  elementary. (However, a key point is that
$U$ need not be wellfounded in general: consider for example
nonprincipal ultrafilters 
over $V_\om$.) If $U$ is wellfounded and extensional,
then by Mostowski's theorem, it is isomorphic
to its (transitive) Mostowski collapse,
and following the usual convention in this case, we then identify
these two. But we will at times need to deal with
ultrapowers without knowing that these properties hold.

In this section we are only actually interested in the case
that the ordinal $\beta$ above is a successor, so from now on,
we restrict to this case.
In order to analyze ultrapowers and the associated embeddings
defined as above,
we will observe that the coding apparatus from \S\ref{subsec:coding}
allows us to represent functions $f:V_\gamma\to V_\beta$
where $\gamma<\beta$ (such as those forming the ultrapower above), and simple 
properties
thereof, in a simple manner. 
That is, although maybe $f\notin V_\beta$, we  
define the \emph{code} of $f$ as
\[ \widetilde{f}=\{\lceil x,y\rceil\bigm|x\in V_{\gamma}
\text{ and }y\in f(x)\}; \]
note $\widetilde{f}\in V_\beta$ (as $\gamma<\beta$ and $\beta$ is a 
successor).
Unravelling the coding above and the flat pairing function,
it is straightforward to write  a 
$\Sigma_0$ formula $\psi$
such that for all such $\beta,\gamma,f$ we have
\[ \all x\in V_{\gamma}\ \all y\in 
V_{\beta-1}\ [y\in f(x)\Leftrightarrow 
V_\beta\sats\psi(\widetilde{f},x,y)].\]
More generally:
\begin{lem}\label{lem:general_Sigma_0_substitution} There is a recursive 
function $\varphi\mapsto\psi_\varphi$
 such that
for each 
$\Sigma_0$ formula $\varphi$,
$\psi_\varphi$ is a $\Sigma_0$ formula,
and for all successor ordinals $\beta>\om$ and ordinals $\gamma<\beta$
and all finite tuples $\vec{f}=(f_0,\ldots,f_{n-1})$ of functions
 $f_i:V_\gamma\to V_\beta$, and all $x\in V_\gamma$ and $z\in V_\beta$,
 \[
V_\beta\sats\varphi(f_0(x),\ldots,f_{n-1}(x),z)\iff 
V_\beta\sats\psi_\varphi(\widetilde{f_0},\ldots,\widetilde{f_1},x,z).\]
\end{lem}

We leave the straightforward proof to the reader.

\begin{dfn}For a transitive structure $M$ and $k\leq\om$,
 $\mathscr{E}_k(M)$ denotes the set of all $\Sigma_k$-elementary maps $j:M\to 
M$. So $\mathscr{E}_\om(M)=\mathscr{E}(M)$.
\end{dfn}

Now suppose $\beta$ is a successor ordinal and $j\in\mathscr{E}_0(V_\beta)$ and 
$j(V_\alpha)=V_{j(\alpha)}$ for each 
$\alpha<\beta$.
Let $\alpha+1<\beta$ and  $s\in V_{j(\alpha)+1}$.
The ultrafilter $\mathscr{F}$ over $V_{\alpha+1}$ \emph{derived 
from $j$ with seed $s$} is defined as follows: For $\sigma\sub V_{\alpha+1}$, 
set
\[ \sigma\in \mathscr{F}\iff s\in j(\sigma).\]
Note that $\mathscr{F}$ is principal iff $s\in\rg(j)$.

For $f:V_{\alpha+1}\to V_\beta$, we needn't have $f\in V_\beta=\dom(j)$,
but we define
\[ j(f):V_{j(\alpha)+1}\to V_\beta \]
to be the function $g$ such that $\widetilde{g}=j(\widetilde{f})$.

Let $U=\Ult(V_\beta,\mathscr{F})$.
Define the  \emph{natural factor map}
$\pi:U\to V_\beta$
by
\[ \pi([f])=j(f)(s).\]
Then $\pi$ is well-defined.
For if $[f]=[g]$
then there is $\sigma\in\mathscr{F}$ such that
\[ \all x\in\sigma\ [f(x)=g(x)], \]
so by Lemma \ref{lem:general_Sigma_0_substitution},
\[ \all x\in j(\sigma)\ [j(f)(x)=j(g)(x)],\]
and since $\sigma\in\mathscr{F}$, therefore $j(f)(s)=j(g)(s)$.
Similarly, $\pi:U\to\rg(\pi)$ is an isomorphism (with respect $\in^U$ and 
$\in$). In particular, in this case, $U$ is wellfounded.
However, without $\AC$, it is not immediate that $U$ is extensional.
That is, suppose $[f]\neq[g]$. To witness extensionality for $[f],[g]$,
we need some $h:V_{\alpha+1}\to V_\beta$ such that $[h]\in^U[f]$ iff 
$[h]\notin^U[g]$; that is, we need $\sigma\in\mathscr{F}$
such that $h(x)\in f(x)\Delta g(x)$ for all $x\in\sigma$ (where $\Delta$ 
denotes symmetric difference). Because $[f]\neq[g]$,
there is indeed $\sigma\in\mathscr{F}$ such that $f(x)\Delta g(x)\neq\emptyset$
for all $x\in\sigma$, but it is not clear whether there is a corresponding 
choice function (even on some smaller $\sigma'\in\mathscr{F}$).

\subsection{Successor rank-to-rank embeddings as ultrapowers}

In this section we sketch an alternate proof of Theorem 
\ref{tm:cumulative_periodicity}, one which is equivalent
to that presented already, but superficially different,
and maybe more standard for set theory. We will also consider
partial elementarity.

\begin{dfn}
Let $\eta$ be even and $j\in\mathscr{E}_0(V_{\eta+2})$ with 
$j(V_{\eta+1})=V_{\eta+1}$. Then $\mu_j$ denotes
the ultrafilter over $V_{\eta+1}$ derived from $j$ with seed $j\rest 
V_\eta$.\footnote{Note that by our flat pairing convention,
$j\rest V_\eta\in V_{\eta+1}$.}
That is,
\[ \mu_j=\{\sigma\sub V_{\eta+1}:j\rest V_\eta\in j(\sigma)\}.\qedhere\]
\end{dfn}

We will again define for all even ordinals $\delta$
a \emph{canonical extension} operation
$k\mapsto k^+$,
with domain $\mathscr{E}_1(V_\delta)$, such that
$k^+:V_{\delta+1}\to V_{\delta+1}$ (but $k^+$ is not claimed
to be elementary in general), and such that $k^+$
is the unique candidate for a $\Sigma_0$-elementary map $\ell:V_{\delta+1}\to 
V_{\delta+1}$ such that $k\sub\ell$ and $\ell(V_\delta)=V_\delta$.
The operation $k\mapsto k^+$, with domain $\mathscr{E}_1(V_\delta)$,
will be definable over $V_{\delta+1}$ without parameters, uniformly in $\delta$
(meaning that there is a formula
$\psi$ such that for all even $\delta$ and
$k\in\mathscr{E}_1(V_\delta)$ and
$x,y\in V_{\delta+1}$, we have
\[ k^+(x)=y\Leftrightarrow V_{\delta+1}\sats\psi(k,x,y),\]
noting  $k\in V_{\delta+1}$ by our flat pairing convention).
The definition of $k\mapsto k^+$ for $k\in\mathscr{E}_1(V_\delta)$,
and proof of its basic properties,
is by induction on $n<\om$, where
$\delta=\lambda+n$ for some limit ordinal $\lambda$.

If $\delta$ is a limit, then $k^+$ is defined as in
Definition \ref{CanonicalExtension}.

Suppose now that $\delta=\eta+2$ where $\eta$ is even.
Let $j\in\mathscr{E}_0(V_{\eta+2})$ with $j(V_{\eta+1})=V_{\eta+1}$; we want to 
define $j^+$ and prove some facts.\footnote{We will end
up seeing that it follows that $j\in\mathscr{E}_1(V_{\eta+1})$.}
Let $\mu=\mu_j$.
Let:
\begin{enumerate}[label=--]
 \item 
$U=\Ult(V_{\eta+2},\mu)$ and 
$i_\mu:V_{\eta+2}\to U$
be the ultrapower map,
\item $\widetilde{U}=\Ult(V_{\eta+3},\mu)$
and $\widetilde{i}_\mu:V_{\eta+3}\to\widetilde{U}$ be the ultrapower map.
\end{enumerate}
We will eventually show that $i_\mu=j$
and $j\sub\widetilde{i}_\mu$, and define $j^+=\widetilde{i}_\mu$.
We don't yet know $U,\widetilde{U}$ are extensional/wellfounded,
so  these ultrapowers are at the ``representation''
level (their elements are  equivalence classes $[f]_\mu$).

Consider the hull
\begin{equation}\label{eqn:H} H=\Hull^{V_{\eta+2}}(\rg(j)\cup\{j\rest 
V_\eta\}),\end{equation}
where
$\Hull^{M}(X)$,
for $X\sub M$, denotes the set of all $x\in 
M$
such that $x$ is definable over $M$ from parameters
in $X$.
The following claim is a typical feature of 
ultrapowers via a measure
derived from an embedding,
although part \ref{item:Ult_iso_Hull} only holds
because $j\rest V_\eta$ encodes enough information,
and for this it is crucial that the canonical
extension $(j\rest V_\eta)^+=j\rest V_{\eta+1}$,
and that this operation is definable over $V_{\eta+2}$ (in fact it is 
definable over 
$V_{\eta+1}$),
a fact we know by induction.

\begin{lem}\label{lem:Ult(V_eta+2,mu)} Recall $U=\Ult(V_{\eta+2},\mu)$
and $H$ is defined in \tu{(}\ref{eqn:H}\tu{)}. We have:
\begin{enumerate}
 \item\label{item:Ult_iso_Hull} $U$
 is extensional and wellfounded; moreover, $U\iso 
H=V_{\eta+2}$.
\item\label{item:i_mu=j} $i_\mu=j$, after we identify $U$ 
with its transitive collapse $V_{\eta+2}$.
\item\label{item:j_Sigma_1-elem} $j:V_{\eta+2}\to V_{\eta+2}$ is 
$\Sigma_1$-elementary.
\end{enumerate}
\end{lem}
\begin{proof}
Part \ref{item:Ult_iso_Hull}: We first show $H=V_{\eta+2}$. As noted above, 
from the parameter $j\rest 
V_\eta$,
$V_{\eta+2}$ can 
compute
$k=(j\rest V_\eta)^+=j\rest V_{\eta+1}$. Now let $x\in 
V_{\eta+2}$.
Then $x\sub V_{\eta+1}$ and
$x=k^{-1}``j(x)$,
and since $j(x)\in\rg(j)$, this suffices.\footnote{Note that
the proof actually shows that 
$V_{\eta+2}=\Hull_{\Sigma_1}^{V_{\eta+2}}(\rg(j)\cup\{j\rest V_\eta\})$,
where $\Hull_{\Sigma_1}^M(X)$ is defined like $\Hull^M(X)$,
except that it only consists of the $y\in M$ such that
for some $\vec{x}\in X^{<\om}$ and $\Sigma_1$ formula $\varphi$,
$y$ is the unique $y'\in M$ such that $M\sats\varphi(\vec{x},y')$.}

Let
$\pi:U\to V_{\eta+2}$ be the factor map
 $\pi([f]^{V_{\eta+2}}_\mu)=j(f)(j\rest V_\eta)$.
By \S\ref{subsec:ultrapowers}, $\pi$ is a well-defined
 $\in$-isomorphism $U\to\rg(\pi)$.
But then $\rg(\pi)=V_{\eta+2}$, because 
given
$x\in V_{\eta+2}$, let
\[ f_x:\mathscr{E}(V_\eta)\to V_{\eta+2} \]
be 
 $f_x(k)=(k^+)^{-1}``x$,
and note
$j(f_x)(j\rest V_\eta)=x$.

Part \ref{item:i_mu=j}:  We have $i_\mu(x)=[c_x]^{V_{\eta+2}}_\mu$.
But note  $\pi\com i_\mu=j$, because
\[ \pi([c_x]^{V_{\eta+2}}_\mu)=j(c_x)(j\rest V_\eta)=c_{j(x)}(j\rest 
V_\eta)=j(x), \]
since $j$ is elementary. But identifying $U$ with $V_{\eta+2}$, 
then $\pi=\id$, so $i_\mu=j$.

Part \ref{item:j_Sigma_1-elem}: Let $\varphi$ be $\Sigma_0$
and $x,y\in V_{\eta+2}$ with
 $V_{\eta+2}\sats\varphi(j(x),y)$.
We have $y=j(f_y)(j\rest V_\eta)$ where 
$f_y$ is as above.
So
\[ V_{\eta+2}\sats\exists k\in V_{\eta+1}\ [\varphi(j(x),j(f_y)(k))].\]
But since $j$ is $\Sigma_0$-elementary and $j(V_{\eta+1})=V_{\eta+1}$,
therefore
\[ V_{\eta+2}\sats\exists k\in V_{\eta+1}\ [\varphi(x,f_y(k))],\]
hence $V_{\eta+2}\sats\exists z\varphi(x,z)$, as desired.
\end{proof}

Having analyzed $j$ as an ultrapower map,
we now consider extending $j$ to 
$V_{\eta+3}$. Recall $\widetilde{U}=\Ult(V_{\eta+3},\mu)$
and $\widetilde{i}_\mu=i^{V_{\eta+3}}_\mu$.

\begin{dfn}
 Let $R\sub \mathscr{E}(V_\eta)\cross V$ be a relation.
 A \emph{$\mu$-uniformization of $R$}
 is a function $f:\mathscr{E}(V_\eta)\to V$ such that
 for  $\mu$-measure one many $k\in\mathscr{E}(V_\eta)$,
if there is $x$ such that $(k,x)\in R$ then $(k,f(k))\in 
R$.
\end{dfn}

\begin{rem}
 Note that the existence of
 $\mu$-uniformizations is a weak kind of choice principle.
\end{rem}

\begin{lem}\label{lem:when_j_extends_to_Sigma_0-elem_j^+}  We have:
 \begin{enumerate}
\item\label{item:Utilde_wfd} $\widetilde{U}$ is  wellfounded.
\item\label{item:Utilde_extensional_charac} The following are equivalent:
\begin{enumerate}[label=\tu{(}\alph*\tu{)}]
\item\label{item:Utilde_extensional} $\widetilde{U}$ is extensional,
\item\label{item:j_extends_Sigma_0} $j$ extends to a $\Sigma_0$-elementary
$\ell:V_{\eta+3}\to V_{\eta+3}$,
\item\label{item:Sigma_0_unif} for all $R\sub\mathscr{E}(V_\eta)\cross 
V_{\eta+2}$,
there is a $\mu$-uniformization of $R$.
\end{enumerate}
\item \label{item:when_extends_to_Sigma_0} If $\ell:V_{\eta+3}\to V_{\eta+3}$
is a $\Sigma_0$-elementary extension of $j$ then
identifying $\widetilde{U}$ with its transitive collapse,
we have $V_{\eta+2}\psub\widetilde{U}\sub 
V_{\eta+3}$ and
$\ell=\widetilde{i}_\mu$ and $\ell(V_{\eta+2})=V_{\eta+2}$.\footnote{The 
arxiv.org:v1
draft of this paper asserted here  ``$\widetilde{U}\psub V_{\eta+3}$,
and in fact $\mu\notin\widetilde{U}$'', but this was an oversight.
If $\ell$ is \emph{fully} elementary, this holds,
by Theorem \ref{EvenUndefinable}. And
the analogous statement holds with $\eta+2$ replaced by a limit;
see Theorem \ref{tm:no_def_Sigma_1_V_delta_to_V_delta}.
We are not sure in general.}
\end{enumerate}
\end{lem}
\begin{proof}
Part \ref{item:Utilde_wfd}: By Lemma \ref{lem:Ult(V_eta+2,mu)},
the part of the ultrapower formed by functions
with codomain $V_{\eta+2}$ is isomorphic to $V_{\eta+2}$.
It follows that $\widetilde{U}$ is wellfounded.

Part \ref{item:Utilde_extensional_charac}:
Suppose $j\sub\ell\in\mathscr{E}_0(V_{\eta+3})$.  We show
$\ell(V_{\eta+2})=V_{\eta+2}$. Clearly $\ell(V_{\eta+2})\sub V_{\eta+2}$, so
we just need
$V_{\eta+2}\sub\ell(V_{\eta+2})$. Let $x\in V_{\eta+2}$.
Then
$x=j(f_x)(j\rest V_\eta)$.
But
\[ V_{\eta+3}\sats\text{``}f_x(k)\in V_{\eta+2}\text{ for all 
}k\in\mathscr{E}(V_\eta)\text{''},\]
which is a $\Sigma_0$ statement of the parameters
$f_x,V_{\eta+2},\mathscr{E}(V_\eta)$, 
and therefore \[ V_{\eta+3}\sats\text{``}\ell(f_x)(k)\in\ell(V_{\eta+2})\text{
for all }k\in\ell(\mathscr{E}(V_\eta))\text{''},\]
but $j\sub\ell$, and it follows that
$x=\ell(f_x)(j\rest V_\eta)\in\ell(V_{\eta+2})$.

We next show that $\widetilde{i}_\mu=\ell$. For we know
$i_\mu=j$ already, so consider $X\in V_{\eta+3}\cut V_{\eta+2}$,
so $X\sub V_{\eta+2}$. Let $x\in V_{\eta+2}$.
Let
\[ D=\{k\in\mathscr{E}(V_\eta)\bigm|f_x(k)\in X\}.\]
Then $x\in\widetilde{i}_\mu(X)$ iff $D\in\mu$
iff $j\rest V_\eta\in j(D)=\ell(D)$ iff (by $\Sigma_0$-elementarity)
 $\ell(f_x)(j\rest V_\eta)\in\ell(X)$
iff $x=j(f_x)(j\rest V_\eta)\in\ell(X)$.

Now let us deduce that \ref{item:Sigma_0_unif} holds.
So let $R\sub\mathscr{E}(V_\eta)\cross V_{\eta+2}$
and let $D$ be the domain of $R$; that is,
\[ D=\{k\in\mathscr{E}(V_\eta)\bigm|\exists x\ [(k,x)\in R]\}.\]
We may assume $D\in\mu$, so $j\rest V_\eta\in j(D)$.
Now $R\in V_{\eta+3}$ and
\[ V_{\eta+3}\sats\all k\in D\ \exists x\in V_{\eta+2}\ [(k,x)\in R].\]
So by $\Sigma_0$-elementarity and since $\ell(V_{\eta+2})\sub V_{\eta+2}$
(in fact we have equality there),
\[ V_{\eta+3}\sats\all k\in \ell(D)\ \exists x\in V_{\eta+2}\ [(k,x)\in 
\ell(R)],\]
and since $D\in\mu$, therefore we can fix $x\in V_{\eta+2}$
such that $(j\rest V_\eta,x)\in\ell(R)$. We claim 
that $f_x$ is
a $\mu$-uniformization of $R$. For  suppose instead that
\[ C=\{k\in\mathscr{E}(V_\eta)\bigm|(k,f_x(k))\notin R\}\in\mu.\]
 Then $j\rest V_\eta\in j(C)=\ell(C)$,
and by
$\Sigma_0$-elementarity,
 $(j\rest V_\eta,\ell(f_x)(j\rest V_\eta))\notin \ell(R)$,
so $(j\rest V_\eta,x)\notin\ell(R)$, a contradiction.

Now assume \ref{item:Sigma_0_unif} holds ($\mu$-uniformization);
 we will show that $\widetilde{U}$ is extensional and $\Sigma_0$-{\L}o\'{s}' 
theorem 
 holds for $\widetilde{U}$, which implies 
 \[ \widetilde{i}_\mu:V_{\eta+3}\to\widetilde{U}\sub V_{\eta+3} \]
 is $\Sigma_0$-elementary, and therefore in fact
 $\widetilde{i}_\mu:V_{\eta+3}\to V_{\eta+3}$ is $\Sigma_0$-elementary.

 For extensionality, let $f,g:\mathscr{E}(V_\eta)\to V_{\eta+3}$
 be such that $[f]\neq[g]$; that is, 
 \[ D=\{k\in\mathscr{E}(V_\eta)\bigm|f(k)\neq g(k)\}\in\mu.\]
 Then define the relation
 \[ R=\{(k,x)\in\mathscr{E}(V_\eta)\cross V_{\eta+2}\bigm|x\in 
f(k)\Delta g(k)\}.\]
Note that for all $k\in D$,  there is $x$ with $(k,x)\in R$. So we can 
$\mu$-uniformize $R$
with some $h:\mathscr{E}(V_\eta)\to V_{\eta+2}$.
Since $\mu$ is an ultrafilter, either (i) for $\mu$-measure one many
$k$, we have $h(k)\in f(k)\cut g(k)$, or (ii) vice versa.
Suppose (i) holds. Then $[h]\in[f]$ and $[h]\notin[g]$,
verifying extensionality for $[f],[g]$.

It follows now that $\widetilde{U}$ is isomorphic
to some subset of $V_{\eta+3}$ (and we already know
$V_{\eta+2}\sub\widetilde{U}$). Now observe
that  the assumed
$\mu$-uniformization is enough for the usual
proof of $\Sigma_0$-{\L}o\'{s}' theorem to go through
(that is, the usual proof of {\L}o\'{s}' theorem, but just with respect 
to $\Sigma_0$ formulas).
It follows as usual that $\widetilde{i}_\mu$
is $\Sigma_0$-elementary as a map $V_{\eta+3}\to\widetilde{U}$,
and hence as a map $V_{\eta+3}\to V_{\eta+3}$, as desired.

Finally suppose that $\mu$-uniformization as in \ref{item:Sigma_0_unif}
fails; we will show that $\widetilde{U}$ is not extensional.
Let $R\sub\mathscr{E}(V_\eta)\cross V_{\eta+2}$
be a counterexample to $\mu$-uniformization.
We have the constant function $c_\emptyset$.
Define $f:\mathscr{E}(V_\eta)\to V_{\eta+3}$
by
\[ f(k)=\{x\bigm|(k,x)\in R\}.\]
Note that $f(k)\neq\emptyset$ for almost all $k$.
So $[f]\neq[c_\emptyset]$. 
But
 there is no $g$ such that $[g]\in[f]$,
 and therefore $\widetilde{U}$ is non-extensional
 with respect to $[f],[c_\emptyset]$.
 
Part \ref{item:when_extends_to_Sigma_0}: We already
saw these things in the proof of part \ref{item:Utilde_extensional_charac}. 
\end{proof}

\begin{dfn}[Canonical extension via ultrapowers]
Let $\eta$ be even.

For $j\in\mathscr{E}_1(V_{\eta+2})$, we define
 $j^+:V_{\eta+3}\to V_{\eta+3}$
as $j^+=\widetilde{i}_{\mu_j}$, as above.

 For $x\in V_{\eta+2}$, $f_x:\mathscr{E}(V_\eta)\to V_{\eta+2}$
is defined $f_x(k)=(k^+)^{-1}``x$ (really $f_x$ depends on $\eta$, but this 
should be clear in context).
\end{dfn}

\begin{rem}\label{rem:proof_2_cumulative_periodicity}
We now reprove Theorem \ref{tm:cumulative_periodicity},
by induction, using the canonical extension $j^+$ just defined.
The argument is essentially as before, so we just give a sketch.
Let $\lambda$ be a limit and $j:V_{\lambda+2}\to V_{\lambda+2}$
be elementary. Let $\mu=\mu_j$. By Lemma
\ref{lem:Ult(V_eta+2,mu)},
$V_{\lambda+2}=\Ult(V_{\lambda+2},\mu)$
and $j=i^{V_{\lambda+2}}_\mu$ is the ultrapower map.

We claim  $j$ is not definable over $V_{\lambda+2}$ from parameters.
For suppose
$j$ is definable over $V_{\lambda+2}$ from $p\in V_{\lambda+2}$.
Then $p\in\rg(j(f_p))$, since
\[ p=[f_p]^{V_{\lambda+2}}_\mu=j(f_p)(j\rest V_\lambda). \]
  One can now argue as in the proof of Theorem
\ref{EvenUndefinable} to reach a contradiction.

Next, if $\ell:V_{\lambda+3}\to V_{\lambda+3}$ is elementary
and $j=\ell\rest V_{\lambda+2}$,
then  $\ell=j^+$ by the preceding lemmas.
But $\mu\in V_{\lambda+3}$, and it is straightforward
to see that the ultrapower map
$j^+=\widetilde{i_\mu}=i^{V_{\lambda+3}}_\mu$ is definable over $V_{\lambda+3}$
from the parameter $\mu$, or equivalently, from
$j$. So $\ell$ is definable as desired.

Now suppose $j:V_{\lambda+4}\to V_{\lambda+4}$
is elementary. Let $\mu=\mu_j$
(the measure derived from $j$ with seed
$j\rest V_{\lambda+2}$).
Then since 
$j\rest V_{\lambda+3}=(j\rest V_{\lambda+2})^+$,
the lemmas give that $\Ult(V_{\lambda+4},\mu)=V_{\lambda+4}$
and $j$ is the ultrapower map, so like before, we get that $j$ is not definable 
from parameters. Etc.
\end{rem}

\subsection{Reinhardt ultrafilters}

Let $\lambda$ be even. One can abstract out a notion of filter which corresponds
precisely to elementary embeddings in $\mathscr{E}(V_{\lambda+2})$, and 
also filters which correspond to embeddings in 
$\mathscr{E}_{n+1}(V_{\lambda+2})$,
for each $n<\om$.
The filters below are over $V_{\lambda+1}$,
but one could consider instead filters over $\mathscr{E}(V_\lambda)$,
identifying $j\in\mathscr{E}(V_\lambda)$ with $\rg(j)$,
and  at a small abuse of notation,
we treat the two interchangeably.
Recall that $\trancl(M)$ denotes the transitive collapse of $M$.

\begin{dfn}[Reinhardt ultrafilters]
Let $\lambda$ be even and  $\mu$ be an ultrafilter over 
$V_{\lambda+1}$. We say that $\mu$ is:
 \begin{enumerate}
  \item \emph{rank-J\'onsson} iff
  $\sigma=\{A: A\elem V_\lambda\text{ and 
  }\trancl(A)=V_\lambda\}\in\mu$,
  \item \emph{fine} iff for each $x\in V_\lambda$, we have
 $\tau_x=\{A: x\in A\sub 
V_\lambda\}\in\mu$,
\item \emph{normal} iff for each $\left<\sigma_x\right>_{x\in V_\lambda}\sub 
\mu$,
the diagonal intersection
\[\Delta_{x\in V_\lambda}\sigma_x= \{A: A\sub V_\lambda\text{ and 
}A\in\sigma_x\text{ for each 
}x\in A\}\in \mu, 
\]
\item \emph{pre-Reinhardt} iff rank-J\'onsson, fine and normal,
\item \emph{$\Sigma_1^{\lambda+2}$-Reinhardt} iff pre-Reinhardt and every
$R\sub V_{\lambda+1}\cross V_{\lambda+1}$ can be $\mu$-uniformized,
\item \emph{$\Sigma_{n+2}^{\lambda+2}$-Reinhardt} iff
 pre-Reinhardt
and every
$R\sub V_{\lambda+1}\cross V_{\lambda+2}$
which is $\Pi_{n}$-definable over $V_{\lambda+2}$ from parameters
can be $\mu$-uniformized,
\item \emph{$\Sigma^{\lambda+2}_\om$-Reinhardt} iff
$\Sigma_{n+1}^{\lambda+2}$-Reinhardt for all $n<\om$.\qedhere
\end{enumerate}
\end{dfn}

 Note
 that if $x\in V_{\lambda+i}$, where $i\leq 1$, then 
$f_x:\mathscr{E}(V_\lambda)\to V_{\lambda+i}$.
\begin{lem}\label{lem:pre-Reinhardt}
Let $\mu$ be a pre-Reinhardt ultrafilter over $V_{\lambda+1}$.
Let $U=\Ult(V_{\lambda+1},\mu)$.
Then
$U$ is extensional, wellfounded and isomorphic to 
$V_{\lambda+1}$. Moreover, $[\id]_\mu=i_\mu``V_\lambda$ and
$[f_x]_\mu=x$ for each $x\in V_{\lambda+1}$.
\end{lem}
\begin{proof} We start by considering $V_\lambda$. Write $[f]=[f]_\mu$.
\begin{clm*} $\Ult(V_\lambda,\mu)=V_\lambda$
and $x=[f_x]$ for each $x\in V_\lambda$.\end{clm*}
\begin{proof}
Given $x,y\in V_\lambda$, we have
\[ ([f_x]\in^U[f_y]\Leftrightarrow x\in y)\text{ and }
([f_x]=^U[f_y]\Leftrightarrow x=y).\]
For by rank-J\'onssonness and fineness, for $\mu$-measure one many 
$k\in\mathscr{E}(V_\lambda)$, we have $x,y\in\rg(k)$,
and for all such $k$, note $f_x(k)=k^{-1}(x)$ and 
$f_y(k)=k^{-1}(y)$. This yields the  stated equivalences.
Now let $f:\mathscr{E}(V_\lambda)\to V_\lambda$.
We claim that there is $x\in V_\lambda$ such that $[f]=^U[f_x]$.
For supposing not, then for each $x\in V_\lambda$,
defining
\[ \sigma_x=\{k\in\mathscr{E}(V_\lambda):f(k)\neq f_x(k)\}, \]
we get $\sigma_x\in\mu$. So the diagonal intersection
\[ \sigma=\{k\in\mathscr{E}(V_\lambda):f(k)\neq f_{x}(k)\text{ 
for all }x\in k``V_\lambda\} \in\mu,\]
so $\sigma\neq\emptyset$. Let 
$k\in\sigma$
and $\bar{x}=f(k)\in V_\lambda$. Let $x=k(\bar{x})$. Then 
$f_x(k)=k^{-1}(x)=\bar{x}=f(k)$, a contradiction.
\end{proof}

Now let $x\in V_{\lambda+1}\cut V_\lambda$ and $y\in V_\lambda$.
The $[f_y]\in^U[f_x]$ iff $y\in x$,
because $y\in\rg(k)$ for $\mu$-almost every $k$.
Note also that $[f_x]\notin^U[f_y]$.
It also easily follows that if $x'\in V_{\lambda+1}\cut V_\lambda$ with $x'\neq 
x$ then $[f_x]\neq [f_{x'}]$ (consider $[f_y]$ for some $y\in x\Delta x'$).

For extensionality, let $f:\mathscr{E}(V_\lambda)\to 
V_{\lambda+1}$ and
$x=\{y\in V_\lambda:[f_y]\in^U[f]\}$.
We claim $[f]=^U[f_x]$.
To see this,
for each $y\in V_\lambda$,
let
\[ \sigma_y=\{k\in\mathscr{E}(V_\lambda):f_y(k)\in f_x(k)\Leftrightarrow
f_y(k)\in f(k)\}.\]
Note $\sigma_y\in\mu$. Let $\sigma\in\mu$ be the diagonal intersection, and 
note $f(k)=f_x(k)$ for each $k\in\sigma$. So $[f]=^U[f_x]$, as desired,
and extensionality follows easily.

The fact that $[\id]=i_\mu``V_\lambda$
 is a straightforward consequence of fineness and normality.
 The rest of the lemma now follows easily.
\end{proof}

\begin{rem}\label{rem:choice_princ}
We now characterize the elements of $\mathscr{E}_{n+1}(V_{\lambda+2})$
as the ultrapower maps given by $\Sigma_{n+1}^{\lambda+2}$-Reinhardt 
ultrafilters,
 and hence the elements of $\mathscr{E}(V_{\lambda+2})$ as the ultrapower 
maps via
$\Sigma_\om^{\lambda+2}$-Reinhardt ultrafilters.
Note that because of the $\mu$-uniformization aspect
of Reinhardt ultrafilters, the theorem shows that
weak choice principles follow from the existence
of appropriate elementary embeddings.\end{rem}

\begin{tm}
Let $\lambda$ be even and $n<\om$. Then:
\begin{enumerate}
 \item \label{item:filter_is_Reinhardt}
If $j\in\mathscr{E}_{n+1}(V_{\lambda+2})$ then
 $\mu_j$ is a
$\Sigma_{n+1}^{\lambda+2}$-Reinhardt ultrafilter
and $j=i^{V_{\lambda+2}}_{\mu_j}$.
\item\label{item:ult_map_is_elem} Let $\mu$ be  a 
$\Sigma_{n+1}^{\lambda+2}$-Reinhardt ultrafilter, $U=\Ult(V_{\lambda+2},\mu)$
and $j:V_{\lambda+2}\to U$ be the ultrapower map $j=i^{V_{\lambda+2}}_\mu$.
Then:
\begin{enumerate}[label=\tu{(}\alph*\tu{)}]
\item\label{item:U=V_lambda+2} $U$ is extensional and 
wellfounded, 
$U=V_{\lambda+2}$,
$\mu=\mu_j$,
$[\id]=j``V_\lambda$ and 
$x=[f_x]=j(f_x)(j`` 
V_\lambda)$ for each $x\in V_{\lambda+2}$.
\item\label{item:Sigma_n+1-elem} $j\in\mathscr{E}_{n+1}(V_{\lambda+2})$.
\end{enumerate}
 \end{enumerate}
\end{tm}
\begin{proof}
Part \ref{item:filter_is_Reinhardt}: Let $\mu=\mu_j$. Rank-J\'onssonness and 
fineness
are straightforward. Consider normality, and fix 
$\vec{\sigma}=\left<\sigma_x\right>_{x\in V_\lambda}\sub\mu$, and let 
$\left<\sigma'_x\right>_{x\in V_\lambda}=j(\vec{\sigma})$.
 Let $B=\Delta_{x\in V_\lambda}\sigma_x$.
We must see that
\[ j`` V_\lambda\in j(B)=\Delta_{x\in V_\lambda}\sigma'_x.\]
But if $y\in j``V_\lambda$ then $y=j(x)$ for some $x\in V_\lambda$,
and $\sigma_x\in\mu$, so
$j``V_\lambda\in j(\sigma_x)=\sigma'_{y}$,
as desired.

Now let $U=\Ult(V_{\lambda+2},\mu)$.
By Lemma \ref{lem:Ult(V_eta+2,mu)}, $U=V_{\lambda+2}$ and
$j=i^{V_{\lambda+2}}_{\mu}$. 
Let us verify that $\mu$ is $\Sigma_1^{\lambda+2}$-Reinhardt.
Let $R\sub V_{\lambda+1}\cross V_{\lambda+1}$ and $D\in\mu_j$
such that for all $k\in D$, there is $x\in V_{\lambda+1}$ with $(k,x)\in 
R$. Then by $\Sigma_1$-elementarity and since $j\rest V_\lambda\in j(D)$,
 there is $x\in V_{\lambda+1}$ with $(j\rest V_\lambda,x)\in j(R)$.
Fix such an $x$. We have  $x=j(f_x)(j\rest V_\lambda)$.
So letting $D'$ be the set of all $k\in D$
such that $(k,f_x(k))\in R$, then $D'\in\mu$, so we are done.

Now suppose $j\in\mathscr{E}_{n+2}(V_{\lambda+2})$
and let $\psi$ be a $\Pi_n$ formula, $p\in V_{\lambda+2}$
and $D\in\mu$, such that for all $k\in D$, there is $x\in V_{\lambda+2}$ with 
$V_{\lambda+2}\sats\psi(p,k,x)$. The assertion
``$\all k\in D\ \exists x\ \psi(p,k,x)$''
is $\Pi_{n+2}$ in parameters $D,p$. So by $\Sigma_{n+2}$-elementarity and since 
$j\rest V_\lambda\in j(D)$,
we can fix $x\in V_{\lambda+2}$ such that
\[ V_{\lambda+2}\sats\psi(j(p),j\rest 
V_\lambda,x).\] Let $D'$ be the set of all $k\in D$
where $V_{\lambda+2}\sats\psi(p,k,f_x(k))$. We claim
 $D'\in\mu$, giving the desired $\mu$-uniformization.
So suppose otherwise. Then $E=\mathscr{E}(V_\lambda)\cut D'\in\mu$,
and $V_{\lambda+2}\sats\all k\in E\ [\neg\psi(p,k,f_x(k))]$.
But then by $\Sigma_{n+1}$-elementarity and since $j\rest V_\lambda\in j(E)$, 
we get $V_{\lambda+2}\sats\neg\psi(j(p),j\rest V_\lambda,x)$,
a contradiction.

Parts \ref{item:ult_map_is_elem}\ref{item:U=V_lambda+2}:
 By Lemma \ref{lem:pre-Reinhardt},
we already know $V_{\lambda+1}=\Ult(V_{\lambda+1},\mu)$
(including extensionality and 
wellfoundedness) and $x=[f_x]$ for all $x\in V_{\lambda+1}$.
Note that it follows that $U=\Ult(V_{\lambda+2},\mu)$ is wellfounded
(though we haven't yet shown extensionality).

Now  $\mu$ is $\Sigma_1^{\lambda+2}$-Reinhardt. Using this, extensionality
is just like in the proof of Lemma \ref{lem:when_j_extends_to_Sigma_0-elem_j^+}.
So we identify $U$ with its Mostowski collapse,
 so $V_{\lambda+1}\sub U\sub V_{\lambda+2}$.
Similarly to extensionality, $\Sigma_0$-{\L}o\'{s}' theorem holds. The 
$\Sigma_1$-elementarity
of $j:V_{\lambda+2}\to U$ follows:
 if $U\sats\exists w\varphi(j(x),w)$ where $\varphi$ is 
$\Sigma_0$, then there is $f$ with $U\sats\varphi(j(x),[f])$, and by
$\Sigma_0$-{\L}o\'{s},
$V_{\lambda+2}\sats\varphi(x,f(k))$ for $\mu$-measure one many $k$,
so $V_{\lambda+2}\sats\exists w\varphi(x,w)$.
And because $[\id]=j``V_\lambda$ by Lemma \ref{lem:pre-Reinhardt},
it is easy to see that $\mu=\mu_j$
(although we haven't shown that $U=V_{\lambda+2}$,
we can still define $\mu_j$ as before).

To see $U=V_{\lambda+2}$,
it suffices to see that $[f_x]=x$ for each $x\in V_{\lambda+2}$,
and for this, given $y\in 
V_{\lambda+1}$, we must see that
 $[f_y]\in^U[f_x]$ iff $y\in x$. 
To see the latter, it suffices to show that $y\in\rg(k^+)$ for $\mu$-measure 
one many $k$,
because for all such $k$, we have $y\in x$ iff
\[ f_k(y)=(k^+)^{-1}(y)\in(k^+)^{-1}``x=f_k(x).\]

Let $D$ be the set of all $k\in \mathscr{E}(V_\lambda)$
such that $y\in\rg(k^+)$. Then since $j$ is $\Sigma_1$-elementary
and $j(V_{\lambda+1})=V_{\lambda+1}$,
$j(D)$ is the set of all $k\in\mathscr{E}(V_\lambda)$
such that $j(y)\in\rg(k^+)$. But 
$j\rest 
V_{\lambda+1}=(j\rest V_\lambda)^+$,
so $j\rest V_\lambda\in j(D)$, so $D$ is $\mu$-measure one,
as desired.

Finally, we already have $[\id]=j``V_\lambda$,
and  $x=[f_x]$ for each $x\in V_{\lambda+2}$.
But then like in the proof of Lemma \ref{lem:Ult(V_eta+2,mu)},
 the factor map $\pi:U\to V_{\lambda+2}$,
defined $\pi([f_x])=j(f_x)(j\rest 
V_\lambda)$, is surjective and in fact the identity,
so $x=j(f_x)(j\rest V_\lambda)$.

Part \ref{item:ult_map_is_elem}\ref{item:Sigma_n+1-elem}:
For $n=0$, this was 
verified above. So suppose $m<\om$ and
$\mu$ is $\Sigma_{m+2}^{\lambda+2}$-Reinhardt;
we show $j$ is $\Sigma_{m+2}$-elementary.
Let $\varphi$ be $\Pi_{m+1}$ and suppose that 
$V_{\lambda+2}\sats\varphi(j(x),y)$. We have $y=[f_y]$.
Let $D$ be the set of all $k\in\mathscr{E}(V_\lambda)$
such that $V_{\lambda+2}\sats\varphi(x,f_y(k))$.
It suffices to see that $D\in\mu$, so suppose 
$E=\mathscr{E}(V_\lambda)\cut D\in\mu$.
Let $\psi$ be a $\Pi_m$ formula
such that
\[ \neg\varphi(u,v)\iff\exists 
w\ \psi(u,v,w).\]
So $V_{\lambda+2}\sats\all k\in E\ \exists 
w\ \psi(x,f_y(k),w)$. Since $\mu$ is $\Sigma_{m+2}^{\lambda+2}$-Reinhardt,
there is $E'\in\mu$ and $g:E'\to V_{\lambda+2}$
such that $V_{\lambda+2}\sats\all k\in E'\ \psi(x,f_y(k),g(k))$.
By induction, $j$ is $\Sigma_{m+1}$-elementary, and as
$y=j(f_y)(j\rest V_\lambda)$
and $j\rest V_\lambda\in j(E')$, we get 
$V_{\lambda+2}\sats\psi(j(x),y,j(g)(j\rest V_\lambda))$,
so $V_{\lambda+2}\sats\neg\varphi(j(x),y))$, contradiction.
\end{proof}

\section{$\Sigma_1$-elementarity at limit 
rank-to-rank}\label{sec:limit_Sigma_1_elementarity}

It is natural to ask whether we can prove a version of
Theorem \ref{tm:cumulative_periodicity} when we assume less than full 
elementarity
of the maps.
Here we focus on the limit case; the successor case is less clear.
It is easy to see that if we only demand $\Sigma_0$-elementarity,
then the embedding can easily be definable from parameters:
\begin{exm}
 Assume $\ZFC$, let $\mu$ be a normal measure
 and $j:V\to\Ult(V,\mu)$ be the ultrapower map,
 and identify $\Ult(V,\mu)$ with transitive
 $M\sub V$. Then note that in fact, $j:V\to V$ is $\Sigma_0$-elementary
 and definable from the parameter $\mu$. So $j$
might even be definable without parameters.\end{exm}

We  now consider the case that $\delta$ is a limit and
$j\in\mathscr{E}_1(V_\delta)$.
We need some more standard set theoretic notions,
but expressed appropriately for the $\ZF$ context.

\begin{dfn} Let $\kappa\in\OR$. We say $\kappa$ is \emph{inaccessible}
iff whenever $\alpha<\kappa$ and $\pi:V_\alpha\to\kappa$,
then $\rg(\pi)$ is bounded in $\kappa$.
The \emph{cofinality} $\cof(\kappa)$ of $\kappa$
is the least $\eta\in\OR$ such that there is a map
$\pi:\eta\to\kappa$ with $\rg(\pi)$ unbounded in $\kappa$.
We say $\kappa$ is \emph{regular} iff $\cof(\kappa)=\kappa$.

A \emph{norm} on a set  $X$ is a surjective function $\pi:X\to\eta$
for some $\eta\in\OR$.
The associated \emph{prewellorder} on $X$ is the relation
$R$ on $X$ given by
$xRy$ iff $\pi(x)\leq\pi(y)$.
One can also axiomatize prewellorders on $X$
as those relations $R$ on $X$ which are linear, total,
reflexive, with wellfounded strict part
(the strict part is the relation $x<_Ry$ iff $[xRy\text{ and }\neg yRx]$).

If $\kappa$ is regular  but non-inaccessible,
and $\alpha\in\OR$ is least such that there is a cofinal
map $\pi:V_\alpha\to\kappa$,
then the \emph{Scott ordertype} of 
$\kappa$, denoted $\scot(\kappa)$,\footnote{This is
an abbreviation of \emph{Sc}ott \emph{o}rder\emph{t}ype.
The second author thanks Asaf Karagila for suggesting
this terminology.}
is the set of all prewellorders of $V_\alpha$ whose ordertype is $\kappa$.
\end{dfn}

\begin{rem}
 Suppose $\kappa$ is regular but not inaccessible,
 and let $\alpha$ be as above and $\pi:V_\alpha\to\kappa$
 be cofinal. Then $\rg(\pi)$ has ordertype
 $\kappa$, as otherwise $\kappa$ is singular. 
 Moreover, $\alpha$ is a successor ordinal.
 For otherwise, by the minimality of $\alpha$,
 we get a cofinal function $f:\alpha\to\kappa$
 by defining $f(\beta)=\sup(\pi``V_\beta)$
 for $\beta<\alpha$, again contradicting regularity.
\end{rem}

\begin{dfn}
 Let $\delta$ be a limit
 and $j\in\mathscr{E}_1(V_\delta)$. For 
$A\sub V_\delta$, 
define $j^+(A)$ just as in Definition \ref{CanonicalExtension}.
 Define $j_0=j$ and for $n\geq 0$ define $j_{n+1}=j^+(j_n)$.
Say $x\in V_\delta$ 
 is \emph{$(j,n)$-stable} iff $j_m(x)=x$ for all $m\in[n,\om)$.
 
 Say that $j$ is \emph{nicely stable} iff either
(i) $\delta$ is inaccessible, or
 (ii) $\delta$ is singular and $j(\cof(\delta))=\cof(\delta)$, or
 (iii) $\delta$ is regular non-inaccessible and 
$j(\scot(\delta))=\scot(\delta)$.

For $j:V_\delta\to V_\delta$ and $A,B\sub V_\delta$,
say $j:(V_\delta,A)\to(V_\delta,B)$ is 
\emph{\tu{(}$\Sigma_n$-\tu{)}elementary}
iff $j$ is ($\Sigma_n$-)elementary in the language $\Ll_{\mathrm{\dot{A}}}$,
with $\dot{A}$ interpreted by the predicates $A,B$ respectively.
\end{dfn}

Before we state the next theorem, we state a corollary in advance:

\begin{cor}
 Let $\delta\in\Lim$ and
 $j\in\mathscr{E}_1(V_\delta)$ be nicely stable. 
 Then $j\in\mathscr{E}(V_\delta)$. In fact,
$j:(V_\delta,A)\to(V_\delta,j^+(A))$
 is fully elementary  for every $A\sub V_\delta$.
\end{cor}

\begin{tm}[An iterate is elementary]\label{tm:an_iterate_is_elementary}
Let $\delta\in\Lim$ and 
$j\in\mathscr{E}_1(V_\delta)$.\footnote{Recall that by 
Lemma \ref{lem:Sigma_1-elem_pres_rank},
$j(V_\alpha)=V_{j(\alpha)}$ for each $\alpha<\delta$.}
 Then:
 \begin{enumerate}
  \item\label{item:j^n_pres_Sigma_1-elem} Every $j_n:V_\delta\to V_\delta$ is 
$\Sigma_1$-elementary; in fact, $j_n:(V_\delta,A)\to(V_\delta,j_n^+(A))$
is $\Sigma_1$-elementary for every $A\sub V_\delta$.
\item\label{item:j^n=} $j_{n+1}=j_n^+(j_n)$.
\item\label{item:x_stable} If $x\in V_\delta$ and $j_n(x)=x$ then $x$ is 
$(j,n)$-stable.
  \item\label{item:alpha_gets_stable} For each $\alpha<\delta$ there is $n<\om$ 
such that $\alpha$
  is $(j,n)$-stable.
  \item\label{item:P_xi_gets_stable} For each $\alpha<\delta$ and $\xi\in\OR$,
  letting $P$ be the set of all prewellorders of $V_\alpha$
of length $\xi$, there is $n<\om$ such that
$P$ is $(j,n)$-stable.
\item\label{item:j_n_nicely_stable_exists} There is $n<\om$ such that $j_n$ is 
nicely stable.
 \item\label{item:j^n_fully_elem} 
Suppose  $j_n$ is nicely stable.
 Then $j_n\in\mathscr{E}(V_\delta)$.
In fact,  for each $A\sub V_\delta$,
the map $j_n:(V_\delta,A)\to(V_\delta,j_n^+(A))$ is fully elementary.
 \end{enumerate}
\end{tm}
\begin{proof}
 For this proof we just write $j(A)$ instead of $j^+(A)$, and $j_n(A)$ instead 
of $j_n^+(A)$, for $A\sub V_\delta$. Note this is unambiguous when $A\in 
V_\delta$.

 Part \ref{item:j^n_pres_Sigma_1-elem}:
 Let $\alpha<\delta$ and $\alpha'=j(\alpha)$ and $j'=j\rest V_\alpha$.
 So $j':V_\alpha\to V_{\alpha'}$ is fully elementary.
 This fact is preserved by $j$, by $\Sigma_1$-elementarity.
 Clearly also $j(j):V_\delta\to V_\delta$, and
 is therefore $\Sigma_0$-elementary with respect to these models.
 But $j(j)$ is also $\in$-cofinal, hence $\Sigma_1$-elementary (with respect 
to 
$\in$).

For the $\Sigma_1$-elementarity of
$j_n:(V_\delta,A)\to(V_\delta,j_n(A))$,
let $x\in V_\delta$ and $\varphi$ be $\Sigma_0$
(in the expanded language), and suppose
\[ (V_\delta,j_n(A))\sats\exists y\ \varphi(j_n(x),y).\]
Let $\alpha<\delta$ be sufficiently large that $x\in V_\alpha$ and
\[ (V_{j_n(\alpha)},j_n(A)\inter V_{j_n(\alpha)})\sats\exists y\ 
\varphi(j_n(x),y).\]
Then by the $\Sigma_1$-elementarity of $j_n$ (just in the language
with $\in$),
\[ (V_\alpha,A\inter V_\alpha)\sats\exists y\ \varphi(x,y), \]
so
$(V_\delta,A)\sats\exists y\ \varphi(x,y)$
as desired.

Part \ref{item:j^n=}:
For $n=0$ this is just the definition.
For $n=1$ note that:
\[ j_2=j(j_1)=j(j(j))=(j(j))(j(j))=j_1(j_1).\]
The rest is similar.
 
 Part \ref{item:x_stable}: If $x=j(x)$ then
 $j(x)=j(j(x))=j(j)(j(x))=j(j)(x)$.
 
 Part \ref{item:alpha_gets_stable}:
 Suppose not and let $\alpha<\delta$ be least otherwise.
 We use the argument in \cite{super_rein}, which is just a slight variant on
 the standard proof of linear iterability.
 For $n<\om$ let
 $A_n=\{\beta<\alpha: j_n(\beta)=\beta\}$.
 So $\alpha=\bigcup_{n<\om}A_n$
 and $\left<A_n\right>_{n<\om}\in V_\delta$.
Note $j(A_n)=\{\beta<j(\alpha): j_{n+1}(\beta)=\beta\}$
and
\[ j(\alpha)=j\left(\bigcup_{n<\om}A_n\right)=\bigcup_{n<\om}j(A_n).\]
But $\alpha<j(\alpha)$ by choice of $\alpha$
and part \ref{item:x_stable}, so $\alpha\in j(A_n)$ for some $n$,
so $j_{n+1}(\alpha)=\alpha$, contradiction.

Part \ref{item:P_xi_gets_stable}: By the above,
there is $n_0$ such that $\alpha$ is $(j,n_0)$-stable.
Now argue as in the previous part from $n_0$ onward,
and using the parameter $\alpha$, define
the collection $P$ of prewellorders of $V_\alpha$
of the form $P=P_\xi$ for some ordinal $\xi$,
with $\xi$ least
such that for no $n\in[n_0,\om)$ is $j_n(P)=P$.
Here $\xi\geq\delta$ is possible.
Note that the notion of \emph{prewellorder} (regarding
relations $R\in V_\delta$)
is simple enough that it is preserved by our
$\Sigma_1$-elementary maps. Likewise,
the lengths of 2 prewellorders can be compared in
a simple enough fashion,
and hence we always have $j_n(P_\xi)=P_{\xi'}$
with some $\xi'_n$. In fact, we get $\xi'_n>\xi$,
since $j_n(P_\zeta)=P_\zeta$
for $\zeta<\xi$. One can now argue for a contradiction much as before.

Part \ref{item:j_n_nicely_stable_exists}: 
By parts \ref{item:alpha_gets_stable} and \ref{item:P_xi_gets_stable}.
 
 Part \ref{item:j^n_fully_elem}: 
 If $\delta$ is 
inaccessible
 then for every $A\sub V_\delta$, 
$(V_\delta,A)\sats\ZF(A)$.\footnote{That is, $\ZF$ augmented
with Collection and Separation for formulas in the language
with $\in$ and $\dot{A}$, and $\dot{A}$ interprets $A$.}
By part \ref{item:j^n_pres_Sigma_1-elem},
 $j:(V_\delta,A)\to(V_\delta,j(A))$ is $\Sigma_1$-elementary.
Therefore a direct relativization of Fact 
\ref{fact:ZF_Sigma_1-elem_cofinal_implies_full_elem}
shows that $j$ is fully elementary in the expanded language.

Now consider the case that $\delta$ is singular and let $\gamma=\cof(\delta)$.
By renaming, we may assume
 $j(\gamma)=\gamma$. Let $A\sub V_\delta$.
 We know
 $j:(V_\delta,A)\to(V_\delta,j(A))$
is $\Sigma_1$-elementary, and must show it is fully elementary.

We begin with $\Sigma_2$-elementarity.
Let $x\in V_\delta$ and $\varphi$ be $\Pi_1$
and suppose that
\[ (V_\delta,j(A))\sats\exists y\varphi(j(x),y),\]
and let $\beta<\delta$ be such that some
$y\in V_{j(\beta)}$
witnesses this.

Suppose first that $\gamma<\delta$;
so we are assuming $j(\gamma)=\gamma$. Let $f:\gamma\to\delta$ be cofinal 
and increasing. For $\xi<\gamma$ let
\[ B_\xi=\{z\in V_\beta:  (V_{f(\xi)},A\inter 
V_{f(\xi)})\sats\varphi(x,z)\}.\]
Then note that
\[ j(B_\xi)=\{z\in V_{j(\beta)}: (V_{j(f(\xi))},j(A)\inter 
V_{j(f(\xi))})\sats\varphi(j(x),z)\}.\]
Therefore $y\in j(B_\xi)$, so in fact
$y\in \left(\bigcap_{\xi<\gamma}j(B_\xi)\right)\neq\emptyset$.
As $\gamma<\delta$, we have $\left<B_\xi\right>_{\xi<\gamma}\in V_\delta$.
Also,
\[ \xi_0<\xi_1\implies B_{\xi_1}\sub B_{\xi_0}.\]
So the same holds of $j(\left<B_\xi\right>_{\xi<\gamma})$,
and since $j(\gamma)=\gamma$, we have $j``\gamma$ cofinal in $j(\gamma)$,
and so letting 
$j(\left<B_\xi\right>_{\xi<\gamma})=\left<B'_\xi\right>_{\xi<\gamma}$,
\[j\left(\bigcap_{\xi<\gamma}B_\xi\right)=\bigcap_{\xi<\gamma}B'_\xi=
\bigcap_{\xi<\gamma} B'_{j(\xi)}=\bigcap_{\xi<\gamma}j(B_\xi)\neq\emptyset.\]
So $\bigcap_{\xi<\gamma}B_\xi\neq\emptyset$.
But letting $z\in\bigcap_{\xi<\gamma}B_\xi$,
note that
$(V_\delta,A)\sats\varphi(x,z)$,
as desired.

Now suppose instead that $\delta$ is regular non-inaccessible.
Define $\left<B_\xi\right>_{\xi<\delta}$
as before, except that now $f(\xi)=\xi$ for $\xi<\delta$. If there 
is $\xi_0<\delta$
such that $B_\xi=B_{\xi_0}$ for all $\xi\in[\xi_0,\delta)$,
then we easily have that $B_{\xi_0}\neq\emptyset$, and any
$z\in B_{\xi_0}$ witnesses $\exists y\varphi(x,y)$
as before. Now suppose there is no such $\xi_0$.
Given $z_0,z_1\in B=\bigcup_{\xi<\delta}B_\xi$,
say that $z_0<^*z_1$ iff there is $\xi<\delta$
such that $z_1\in B_\xi$ but $z_0\notin B_\xi$.
Then $<^*$ is a prewellorder on $B$,
and $<^*$ is in $V_\delta$, and because
$\gamma=\delta$ and there is no $\xi_0$ as above,
$\delta$ is the the ordertype of $<^*$. 
So let $P=\scot(\delta)$, so by assumption $j(P)=P$, which easily gives
that $j({<^*})$ also has ordertype $\delta$.
The function $z\mapsto B_{\rank^*(z)}$, with domain $B$,
and where $\rank^*(z)$ is the ${<^*}$-rank of $z$,
is also in $V_\delta$. But then we can
argue as before to show $\bigcap_{\xi<\delta}B_\xi\neq\emptyset$,
which suffices, also as before.

So we have  $\Sigma_2$-elementarity (with respect to an
arbitrary $A\sub V_\delta$).
Now suppose
we have $\Sigma_k$-elementarity where $k\geq 2$.
Define the theory
\[ T=T^A_{k-1}=\Th_{\Sigma_{k-1}}^{(V_\delta,A)}(V_\delta);\]
this denotes the theory consisting of all pairs $(\varphi,x)$ such 
that
$\varphi$ is a $\Sigma_{k-1}$ formula and
$(V_\delta,A)\sats\varphi(x)$. The $\Sigma_k$-elementarity of $j$ gives:

\begin{clm} 
$j(T)=\Th_{\Sigma_{k-1}}^{(V_\delta,j(A))}(V_\delta)$.\end{clm}
\begin{proof}
Given $\alpha<\delta$, we have
\[ (V_\delta,A)\sats\all x\in 
V_\alpha\ [\all\ \Sigma_{k-1}\text{ formulas }\varphi\text{ 
of }\Ll_{\dot{A}}\ 
[\varphi(x)\iff (\varphi,x)\in T\inter V_\alpha]], \]
which is a $\Pi_k$ assertion of parameter $(V_\alpha,T\inter V_\alpha)$, which
therefore lifts to $(V_\delta,j(A))$ regarding the parameter 
$(V_{j(\alpha)},j(T)\inter V_{j(\alpha)})$.\end{proof}

So by what we have proved above, but with $(A,T)$
replacing $A$, we have that $j$ is $\Sigma_2$-elementary
as a map
\begin{equation}\label{eqn:j_Sigma_2-elem_for_T} j:(V_\delta,(A,T))\to 
(V_\delta,(j(A),j(T))).\end{equation}

Now let $\varphi$ be $\Sigma_{k-1}$ and suppose that
\[ (V_\delta,j(A))\sats\exists y\all z\ [\varphi(j(x),y,z)];\]
equivalently,
\[ (V_\delta,(j(A),j(T)))\sats\exists y\all z\ [(\varphi,(j(x),y,z))\in j(T)].\]
By the $\Sigma_2$-elementarity of $j$ with respect to the structures in line
(\ref{eqn:j_Sigma_2-elem_for_T}) above, therefore
\[ (V_\delta,(A,T))\sats\exists y\all z\ [(\varphi,(x,y,z))\in T];\]
equivalently,
$(V_\delta,A)\sats\exists y\all z\ [\varphi(x,y,z)]$,
as desired.
\end{proof}

Using the preceding theorem, we now improve
on Theorem \ref{LimitUndefinable}:
\begin{tm}\label{tm:no_def_Sigma_1_V_delta_to_V_delta}
 Let $j\in\mathscr{E}_1(V_\delta)$ where $\delta\in\Lim$.
 Then: 
 \begin{enumerate}
  \item\label{item:Sigma_1_elem_j_not_def_from_params} $j$ is not definable
 from parameters over $V_\delta$.
\item\label{item:Sigma_1_elem_j_not_in_own_Ult} 
There is no $(a,f)$ with $a\in V_\delta$ and $f:V_\delta\to 
V_{\delta+1}$ and
 $j^+(f)(a)=j$.
 \end{enumerate}
\end{tm}

\begin{rem}The reader familiar with extenders will note that
in the
proof of part \ref{item:Sigma_1_elem_j_not_in_own_Ult},
we are considering $\Ult(V_{\delta+1},E)$ where $E$ is the $V_\delta$-extender 
derived from $j$. As before, we can represent functions $f:V_\delta\to 
V_{\delta+1}$ via relations $\sub V_\delta\cross V_\delta$, and hence with 
elements of $V_{\delta+1}$, and when $j\in\mathscr{E}_1(V_\delta)$,
one gets
$j^+(f):V_\delta\to V_{\delta+1}$, making sense of the statement of part 
\ref{item:Sigma_1_elem_j_not_in_own_Ult} above.
\end{rem}
\begin{proof}
Part \ref{item:Sigma_1_elem_j_not_def_from_params}: Suppose otherwise. Then by 
Theorem \ref{tm:an_iterate_is_elementary},
there is $n<\om$ such that $j_n:V_\delta\to V_\delta$
is fully elementary, and since $j$ is definable from parameters
over $V_\delta$, so is $j_n$. This contradicts Theorem \ref{LimitUndefinable}.

Part \ref{item:Sigma_1_elem_j_not_in_own_Ult}:
 Suppose otherwise and fix a counterexample $(j,a,f)$. Then for each 
$n<\om$, $(j_n,a_n,f_n)$ is also a counterexample,
 where $(a_0,f_0)=(a,f)$ and $(a_{k+1},f_{k+1})=(j_k(a_k),j_k^+(f_k))$.
 (For note that one can apply $j$ to each initial segment of the sets 
corresponding to the equation $j^+(f)(a)=j$, and their union yields 
$j_1^+(j^+(f))(j(a))=j_1$, so $j_1^+(f_1)(a_1)=j_1$. Etc.)
So by Theorem \ref{tm:an_iterate_is_elementary}, we may
assume  $j$ is nicely stable. Also by
\ref{tm:an_iterate_is_elementary}, it follows that $j^+:V_{\delta+1}\to 
V_{\delta+1}$ is $\Sigma_0$-elementary. Let $\mathscr{I}$ be the  set of 
functions $g:V_\delta\to V_{\delta+1}$. We have
\[ V_{\delta+1}=\{j^+(g)(a): g\in\mathscr{I}\},\]
because if $x\in V_{\delta+1}$ then $x=j^{-1}``j^+(x)$,
so letting $g(u)=f(u)^{-1}``x$ (where $j^+(f)(a)=j$),
we get $x=j^+(g)(a)$. It follows that $j^+$ is $\Sigma_1$-elementary.
For let $\varphi$ be $\Sigma_0$ and suppose 
$V_{\delta+1}\sats\varphi(j^+(x),y)$. Let $g:V_\delta\to V_{\delta+1}$
be such that $j^+(g)(a)=y$.  Note that there is a formula $\psi$
such that for all $u\in V_{\delta+1}$ and 
$h:V_\delta\to V_{\delta+1}$ and $b\in V_\delta$,
we have $V_{\delta+1}\sats\varphi(u,h(b))$ iff $(V_\delta,u,h)\sats\psi(b)$
(where as before, we code $h$ naturally with a relation $\sub V_\delta\cross 
V_\delta$, and $\psi$ has predicates referring to $u,h$). 
But then since
\[ j:(V_\delta,x,g)\to(V_\delta,j^+(x),j^+(g)) \]
is elementary
and $(V_\delta,j^+(x),j^+(g))\sats\exists b\psi(b)$
(as witnessed by $b=a$), we have
$(V_\delta,x,g)\sats\exists b\psi(b)$, so there is $b\in V_\delta$
such that $V_{\delta+1}\sats\varphi(x,g(b))$.

Now if $\delta$ is singular, let $p=\cof(\delta)$,
and if $\delta$ is regular non-inaccessible, let $p=\scot(\delta)$,
and otherwise let $p=\emptyset$. Let $\kappa_0$
be the least critical point of all $k\in\mathscr{E}_1(V_\delta)$
such that $k(p)=p$ and $k=k^+(h)(c)$ for some $c\in V_\delta$ and 
$h:V_\delta\to  V_{\delta+1}$. Fix $j_0,h_0,c_0$ witnessing the choice of 
$\kappa_0$.
By the preceding discussion, $j_0\in\mathscr{E}(V_\delta)$ and 
$j_0^+\in\mathscr{E}_1(V_{\delta+1})$.
We have $p\in\rg(j_0^+)$, but $\kappa_0\notin\rg(j_0^+)$.

Let $\eta=j_0(\kappa_0)$. 
Then $V_{\delta+1}\sats$``there are $k,\mu,h,c$
such that $k\in\mathscr{E}_1(V_\delta)$ and $\crit(k)=\mu<\eta$ and $k(p)=p$ 
and $h:V_\delta\to V_{\delta+1}$ and $c\in V_\delta$ and $k^+(h)(c)=k$''
(as witnessed by $j_0,\kappa_0,h_0,c_0$). Since 
$j_0(p,\delta,\kappa_0)=(p,\delta,\eta)$ and by the $\Sigma_1$-elementarity of 
$j_0^+$,  we can fix some such $\mu\in\rg(j_0)$. But 
note $\kappa_0\leq\mu<\eta$, by the minimality of $\kappa_0$, 
contradiction.
\end{proof}

Many of the arguments applied in this section
to rank-into-rank embeddings, also apply more generally,
and in particular to embeddings consistent with $\ZFC$:

\begin{tm}\label{tm:j:V_eta_to_V_delta}
 Let $\eta<\delta$ be limit ordinals
 and $j:V_\eta\to V_\delta$ be $\Sigma_1$-elementary and $\in$-cofinal.
 Then:
 \begin{enumerate}
  \item\label{item:fully_elem_implies_not_def_from_params} If $j$ is fully 
elementary then $j$ is not definable over $V_\delta$ from parameters.
\item\label{item:fixing_mu_implies_fully_elem} If $\mu=\cof(\eta)<\eta$ and 
$j(\mu)=\mu$ 
then for every $A\sub V_\eta$, defining
\[ j(A)=\bigcup_{\beta<\eta}j(A\inter 
V_\beta),\]
the map
$j:(V_\eta,A)\to j(V_\delta,j(A))$
is fully elementary.
 \end{enumerate}
\end{tm}
\begin{proof}
Part \ref{item:fixing_mu_implies_fully_elem}:
This is almost the same as the proof of the corresponding fact in the singular 
case of Theorem \ref{tm:an_iterate_is_elementary} part \ref{item:j^n_fully_elem} 
(note we have assumed
that $j$ is $\in$-cofinal, which is important).

Part \ref{item:fully_elem_implies_not_def_from_params}:
 Suppose not. Then $\delta$ is singular, definably from parameters over 
$V_\delta$, as witnessed by $j\rest\eta:\eta\to\delta$. Let 
$\mu=\cof(\delta)=\cof(\eta)$. Using the elementarity of $j$,
it easily follows that there is $n<\om$ such that both $V_\eta$
and $V_\delta$ satisfy ``There is a  function 
$k:\mu\to\OR$ which is $\Sigma_n$-definable from parameters,
and $\mu$ is least such'', and $j(\mu)=\mu$. Note that it also follows
that $\mu$ is definable over $V_\delta$ without parameters.

Now fix a formula $\varphi$ and $p\in V_\delta$ such that
$j(x)=y$ iff $V_\delta\sats\varphi(p,x,y)$.
For $q\in V_\delta$ let $j_q=\{(x,y)\bigm|V_\delta\sats\varphi(q,x,y)\}$.
Say $q$ is \emph{good} iff there is a limit $\eta'<\delta$
such that $j_q:V_{\eta'}\to 
V_\delta$ is $\Sigma_1$-elementary
and $\in$-cofinal and $j_q(\mu)=\mu$.
By part \ref{item:fixing_mu_implies_fully_elem}, if $q$ is good then $j_q$
is fully elementary.
Then the least critical point among all good $j_q$,
is definable over $V_\delta$ without 
parameters, which leads to the usual contradiction.
\end{proof}

Of course in the situation above, the iterates $j_n$ of $j$ are not well-defined
(at least not in their earlier form), so we have not ruled out the possibility
of $j:V_\eta\to V_\delta$ which is $\Sigma_1$-elementary and $\in$-cofinal
with $j(\mu)>\mu$, which is definable from parameters.
The following theorem, due to Andreas Lietz and the second author,
shows that if a Reinhardt cardinal exists then it is at times
necessary to pass from $j$ to $j_n$ to secure full elementarity:\footnote{The 
second author
initially noticed the $n=1$ example, then Lietz generalized this
to $n>1$ via basically the method at the end of the proof,
but from a stronger  assumption to secure fixed points, and then the second 
author observed the claim on fixed points, leading to the version here.}
\begin{tm}[Lietz, S.] Suppose $j\in\mathscr{E}(V_{\lambda^+})$ 
 where $\lambda=\kappa_\om(j)$.
 Then for each $n<\om$ there is a limit $\delta<\lambda^+$
 such that $j``\delta\sub\delta$ and $k=j\rest 
V_\delta\in\mathscr{E}_1(V_\delta)$, but 
$k=k_0,k_1,\ldots,k_n\notin\mathscr{E}_2(V_\delta)$.
\end{tm}
\begin{proof}
First consider $n=0$.
 Let $\kappa=\crit(j)$ and $\delta=\lambda+\kappa$
and $k=j\rest V_\delta$. Since $j(\lambda)=\lambda$
and $j\rest\kappa=\id$, we have $k:V_\delta\to V_\delta$,
and clearly $k$ is $\in$-cofinal and $\Sigma_0$-elementary,
hence $\Sigma_1$-elementary. But consider the $\Pi_2$ formula
\[ 
\varphi(\dot{\kappa},\dot{\lambda})=\text{``}
\all\alpha<\dot{\kappa}\ 
\exists\xi\in\OR\ [\xi=\dot{\lambda}+\alpha]\text{''}.\]
Then $V_{\lambda+\kappa}\sats\varphi(\kappa,\lambda)$,
but $V_{\lambda+\kappa}\sats\neg\varphi(j(\kappa),j(\lambda))$;
that is, $V_{\lambda+\kappa}\sats\neg\varphi(j(\kappa),\lambda)$, since
$\alpha=\kappa<j(\kappa)$, but $\lambda+\kappa\not\in V_{\lambda+\kappa}$.
For this example, $k_1(\kappa)=\kappa=\cof(\lambda+\kappa)$,
so $k_1$ is fully elementary, by Theorem \ref{tm:an_iterate_is_elementary}.

Now let $n$ be arbitrary.

\begin{clm*}$j$ has $\lambda^+$-many fixed points $<\lambda^+$.\end{clm*}
\begin{proof} Let $F_n=\{\alpha<\lambda^+:j_n(\alpha)=\alpha\}$.
By Theorem 
\ref{tm:an_iterate_is_elementary}, $\lambda^+=\bigcup_{n<\om}F_n$.
The ordertypes $\alpha_n$ of the $F_n$ are then either unbounded in $\lambda^+$,
or some $\alpha_n=\lambda^+$,
since otherwise one easily constructs a surjection $\pi:\lambda\to\lambda^+$
(consider the uncollapse maps $\pi_n:\alpha_n\to F_n$). Now $F_0$ is 
unbounded in $\lambda^+$.
For suppose not, and let $\sup(F_0)<\beta_0<\lambda^+$.
Let $\pi_0:\lambda\to\beta_0$ be a surjection.
Let $\pi_{n+1}=j(\pi_n)$ and 
$\beta_{n+1}=\rg(\pi_{n+1})=j(\beta_n)$.
From $\left<\pi_n\right>_{n<\om}$
we get a surjection $\lambda\to\beta=\sup_{n<\om}\beta_n$.
Therefore $\beta<\lambda^+$, but note $\cof(\beta)=\om$,
so $j(\beta)=\beta$,  contradicting the choice of $\beta_0$.
Now $\alpha_0=\lambda^+$.\footnote{We don't 
 know that $\lambda^+$ is regular; the first author has results
 in regard to this. So we can't just use the fact that $F_0$ is unbounded in 
$\lambda^+$ here.}
For suppose not.
Then note  $\alpha_{n+1}=\sup j``\alpha_n=\sup j_n``\alpha_n$
(using that $F_n$ is cofinal in $\lambda^+$).
Then letting $\alpha_0<\eta\in F_0$,
note $\alpha_n<\eta$ for all $n<\om$, a contradiction.
\end{proof}

Now let $\delta$ be the supremum of the
first $\crit(j_n)$ fixed points of $j$ which are $>\lambda$.
Then $j``\delta\sub\delta$, so $k=j\rest V_\delta\in\mathscr{E}_1(V_\delta)$.
Let $W$ be a wellorder of $\lambda$ in ordertype $\delta$
(note $\lambda<\delta<\lambda^+$, so $W$ exists).
Then
\begin{equation}\label{eqn:stmt} V_\delta\sats\text{``every proper segment  of
}W\text{ has ordertype some }\alpha\in\OR\text{''}.
\end{equation}
But for $m\leq n$, $k_m(W)$ is a wellorder of $k_m(\lambda)=\lambda$ in 
ordertype some $\delta'_m$,
and $\delta<\delta'_m$,
because (i) the ordertype of $W$ is $\leq$ that of 
$k_m(W)$,
and (ii) $\cof(W)=\crit(k_n)$,
so $\cof(k_m(W))=k_m(\crit(k_n))=\crit(k_{n+1})$.
Since $\delta<\delta'_m$,
$V_\delta$ does not satisfy line (\ref{eqn:stmt}) with $W$ replaced by
$k_m(W)$, so $k_m$ is not $\Sigma_2$-elementary.
\end{proof}

\section{Which ordinals are large enough?}\label{sec:large_enough}

We said in the introduction that if an ordinal $\eta$ is large enough,
then $V_{\eta+183}$ and $V_{\eta+184}$ are very different
from each other. Of course, we have seen that there are such
differences assuming there is an elementary $j:V_{\eta+184}\to V_{\eta+184}$.
So we could take this as the definition of ``large enough'',
but then the term is not very natural, because
then it needn't be that $\eta+1$ is also ``large enough''.
To get a good notion of ``large enough'', we assume that there is a Reinhardt 
cardinal. Let then $j:V\to V$ be elementary with $\kappa_{\om}(j)$
minimal. Then we say that $\eta$ is ``large enough''
iff $\eta\geq\kappa_{\om}(j)$.
Below, $\ZF(j)$ denotes the Zermelo Fr\"ankel axioms
in the language $\mathscr{L}_{j}$ with symbols $\in,j$,
augmented with Collection and Separation for all formulas
in $\mathscr{L}_j$.
Under this theory, we can assert that ``$j:V\to V$ is elementary''
with the single formula ``$j:V\to V$ is $\Sigma_1$-elementary'',
by Fact \ref{fact:ZF_Sigma_1-elem_cofinal_implies_full_elem}.
The following 
theorem was mentioned to the first author by Koellner a few years ago,
but may be folklore.
There are some further related things in
\cite{reinhardt_non-definability}:

\begin{tm}[Folklore?]\label{tm:lambda_rank-Berkeley}
 Assume $\ZF(j)$ and $j:V\to V$ is elementary (non-identity).
 Let $\lambda=\kappa_\om(j)$.
 Then for all $\alpha\geq\lambda$ and all $\eta<\lambda$, there is an elementary
 $k:V_\alpha\to V_\alpha$ such that $\crit(k) > \eta$ and \(\kappa_\omega(k) = \lambda\).
\end{tm}
\begin{proof}
Suppose not and let $(\eta,\alpha)$ be the lexicographically
least counterexample. Then $(\eta,\alpha)$ is definable
from the parameter $\lambda$, and hence fixed by $j$.
But then $j(\alpha)=\alpha$,
so $j\rest V_\alpha:V_\alpha\to V_\alpha$,
and $j(\eta)=\eta<\lambda$, so $\eta<\crit(j)=\crit(j\rest V_\alpha)$,
so $j\rest V_\alpha$ contradicts the choice of $(\eta,\alpha)$.
\end{proof}

So above $\lambda=\kappa_\om(j)$, the cumulative hierarchy
is periodic the whole way up.

\begin{rem}For the reader familiar with \cite{woodin_koellner_bagaria},
note that the property stated of $\lambda=\kappa_\om(j)$
in the theorem above is just that of a Berkeley cardinal
(see \cite{woodin_koellner_bagaria})
\emph{with respect to rank segments of $V$}
(except that we have also stated it for $V_\lambda$ itself,
although $\lambda\notin V_\lambda$).
One could call such a $\lambda$ a \emph{rank-Berkeley cardinal}.
Note that unlike Reinhardtness, rank-Berkeleyness is first-order. 
If there is a Reinhardt, then which is less, 
the least Reinhardt or the least rank-Berkeley?
If $j:V\to V$ and $\lambda=\kappa_\om(j)$ is the least
rank-Berkeley, then note that for every $k:V\to V$ with $\crit(k)<\lambda$, we 
have 
$\lambda_{\om,k}=\lambda$.
In particular, if $\kappa$ is super Reinhardt
then the least rank-Berkeley is ${<\kappa}$.
We show next that the least rank-Berkeley being below the least Reinhardt,
has consistency strength beyond that of a Reinhardt.

We remark that  arguing further as above shows that every
rank-Berkeley is $\HOD$-Berkeley. Can be/is the least $\HOD$-Berkeley
${<}$ the least rank-Berkeley?
\end{rem}

\begin{tm}
 Suppose $(V,j)\sats\ZF(j)$ and $j:V\to V$, and let $\kappa=\crit(j)$
 and $\lambda=\kappa_{j,\om}$, and suppose the least rank-Berkeley is
$\delta<\lambda$.
Let $\mu_j$ be the normal measure over $\kappa$ derived from $j$.
Then $\delta<\kappa$ and there is $\kappa'<\delta$
such that for $\mu_j$-measure one many $\gamma<\kappa$,
$(V_\gamma,V_{\gamma+1})\sats$``$\kappa'$ is a Reinhardt cardinal''.
\end{tm}
\begin{proof}
Suppose $\delta<\lambda$ is rank-Berkeley, so $\delta<\kappa$.
Then there is $k:V_\kappa\to V_\kappa$ which is elementary
and non-identity.
 Let $\kappa'=\crit(k)$. Then $\kappa$ is inaccessible
 and $(V_\kappa,V_{\kappa+1})\sats\ZF_2+$``$\kappa'$ is Reinhardt,
 as witnessed by $k$''.
Since $\kappa=\crit(j)$, the theorem follows routinely.
\end{proof}

\begin{cor}
 Suppose 
 $\ZF(j)+\text{``}j:V\to V\text{''}$
 is consistent.
 Then so is
 \[ \ZF(j)+\text{``}j:V\to V\text{''}+\text{``}\kappa_\om(j)\text{
 is the least rank-Berkeley}\text{''}.\]
\end{cor}

This also gives that $\lambda=\kappa_{\om}(j)$ can be definable over $V$ without
parameters. But there is anyway another way to see that
$j:V\to V$ with 
$\lambda$ non-definable is stronger than just $j:V\to V$.
For since $\lambda$ is a limit of inaccessibles,
if $\lambda$ is non-definable, then
$V$ has inaccessibles  $\delta>\lambda$,
and taking the least such, $j(\delta)=\delta$, so
we get $(V_\delta,V_{\delta+1})\sats\ZF_2+$``There is a Reinhardt''
(actually the latter holds for every inaccessible $\delta>\lambda$,
since $j_n(\delta)=\delta$ for some $n$).
\section{Questions and related work}

In \S\ref{sec:limit_Sigma_1_elementarity} we 
ruled out the definability of $\Sigma_1$-elementary embeddings $j:V_\delta\to 
V_\delta$
for $\delta$ a limit. Note that we also observed
that for $\delta$ even, $\Sigma_1$-elementary maps
$j:V_{\delta+1}\to V_{\delta+1}$ are always
definable from the parameter $j\rest V_\delta$.
But what about  partially elementary maps
$V_{\delta+2}\to V_{\delta+2}$?
Can they be definable from parameters over $V_{\delta+2}$?
If so, what can one say about the complexity of the definition
in relation to the degree of elementarity?

One can also generalize the notion of ``definable from parameters''
to allow higher order definitions, such as looking in $L(V_\delta)$.
If $\delta$ is a limit and $L(V_\delta)\sats$``$\cof(\delta)>\om$''
then $L(V_\delta)$ has no elementary $j:V_\delta\to V_\delta$
(see \cite{extenders_ZF}; the case that $\delta$
is inaccessible was established earlier by the first author).
There is a little on the cofinality
$\om$ case in \cite{extenders_ZF},
but this case is much more subtle.

The existence of the canonical extension $j^+$ of an embedding $j:V_\lambda\to 
V_\lambda$ for limit $\lambda$ is of fundamental importance
to the analysis of $I_0$; see for example \cite{woodin_sem2}. But
this is now naturally generalized to all even $\lambda$.
It turns out that much of the $I_0$ theory generalizes
in turn, and this is one of the topics of \cite{goldberg_even_numbers}.

Of course a significant question looming over this work
is whether embeddings of the form we are considering can even exist.
Some recent progress in this regard, establishing the
consistency of $\ZF+j:V_{\lambda+2}\to V_{\lambda+2}$
relative to $\ZFC+I_0$, is the topic of \cite{con_lambda_plus_2}.
 \bibliographystyle{plain}
\bibliography{./bibliography_rank_embedding_def}

\begin{thebibliography}{10}

\bibitem{sarg_apt_jonsson_V_to_V}
Grigor~Sargsyan Arthur~Apter.
\newblock Jonsson-like partition relations and {$j:V\to V$}.
\newblock {\em Journal of Symbolic Logic}, 69(4), 2004.

\bibitem{short_note_vlc}
David Asper\'{o}.
\newblock A short note on very large cardinals (without choice).
\newblock Available at https://archive.uea.ac.uk/~bfe12ncu/notes.html.

\bibitem{woodin_koellner_bagaria}
Joan Bagaria, Peter Koellner, and W.~Hugh Woodin.
\newblock Large cardinals beyond choice.
\newblock {\em Bulletin of Symbolic Logic}, 25, 2019.

\bibitem{cutolo_structure}
Raffaella Cutolo.
\newblock Berkeley cardinals and the structure of {$L(V_{\delta+1})$}.
\newblock {\em The Journal of Symbolic Logic}, 83(4), 2019.

\bibitem{cutolo_cofinality}
Raffaella Cutolo.
\newblock The cofinality of the least {B}erkeley cardinal and the extent of
  dependent choice.
\newblock {\em Mathematical Logic Quarterly}, 65(1), 2019.

\bibitem{dimonte}
Vincenzo Dimonte.
\newblock {$I_0$} and rank-into-rank axioms.
\newblock {\em Bollettino dell'Unione Matematica Italiana}, 11:315--361, 2018.

\bibitem{goldberg_even_numbers}
Gabriel Goldberg.
\newblock Even ordinals and the {K}unen inconsistency.
\newblock ar{X}iv:2006.01084, 2020.

\bibitem{gen_kunen_incon}
Joel~David Hamkins, Greg Kirmayer, and Norman~Lewis Perlmutter.
\newblock Generalizations of the {K}unen inconsistency.
\newblock {\em Annals of Pure and Applied Logic}, 163(12), 2012.

\bibitem{critical_cardinals}
Yair Hayut and Asaf Karagila.
\newblock Critical cardinals.
\newblock {\em Israel journal of mathematics}, 236(1):449--472, 2020.
\newblock ar{X}iv: 1805.02533.

\bibitem{kanamori}
Akihiro Kanamori.
\newblock {\em The higher infinite: large cardinals in set theory from their
  beginnings}.
\newblock Springer monographs in mathematics. Springer-Verlag, second edition,
  2005.

\bibitem{kunen_no_R}
Kenneth Kunen.
\newblock Elementary embeddings and infinitary combinatorics.
\newblock {\em Journal of Symbolic Logic}, 36(3), 1971.

\bibitem{kunen_set_theory_2011}
Kenneth Kunen.
\newblock {\em Set Theory}.
\newblock College Publications, 2nd edition, 2011.

\bibitem{mosch}
Yiannis~N. Moschovakis.
\newblock {\em Descriptive set theory}.
\newblock North-Holland, 1980.

\bibitem{extenders_ZF}
Farmer Schlutzenberg.
\newblock Extenders under {ZF} and constuctibility of rank-to-rank embeddings.
\newblock ar{X}iv:2006.10574.

\bibitem{con_lambda_plus_2}
Farmer Schlutzenberg.
\newblock On the consistency of {ZF} with an elementary embedding from
  {$V_{\lambda+2}$} into {$V_{\lambda+2}$}.
\newblock ar{X}iv:2006.01077, 2020.

\bibitem{reinhardt_non-definability}
Farmer Schlutzenberg.
\newblock Reinhardt cardinals and non-definability.
\newblock ar{X}iv: 2002.01215v1, 2020.

\bibitem{super_rein}
Farmer Schlutzenberg.
\newblock A weak reflection of {R}einhardt by super {R}einhardt cardinals.
\newblock ar{X}iv: 2005.11111, 2020.

\bibitem{suzuki_no_def_j}
Akira Suzuki.
\newblock No elementary embedding from {$V$} into {$V$} is definable from
  parameters.
\newblock {\em Journal of Symbolic Logic}, 64(4), 1999.

\bibitem{usuba_ls}
Toshimichi Usuba.
\newblock Choiceless {L}\"owenheim-{S}kolem property and uniform definability
  of grounds.
\newblock ar{X}iv: 1904.00895, 2019.

\bibitem{woodin_sem2}
W.~Hugh Woodin.
\newblock Suitable {E}xtender {M}odels ii: {B}eyond $\om$-huge.
\newblock {\em Journal of Mathematical Logic}, 11(2), 2011.

\end{thebibliography}

\pagebreak

\end{document}